\documentclass[12pt]{amsart}

\usepackage[text={400pt,650pt},centering]{geometry}

\usepackage{color}
\usepackage{esint,amssymb,bbold}
\usepackage{graphicx}
\usepackage{MnSymbol}
\usepackage{mathtools}
\usepackage[colorlinks=true, pdfstartview=FitV, linkcolor=blue, citecolor=blue, urlcolor=blue,pagebackref=false]{hyperref}
\usepackage{microtype}

\newtheorem{theorem}{Theorem}
\newtheorem{proposition}[theorem]{Proposition}
\newtheorem{lemma}[theorem]{Lemma}
\newtheorem{corollary}[theorem]{Corollary}

\theoremstyle{definition}

\newcommand{\cref}[1]{Corollary~\ref{c.#1}}

\numberwithin{equation}{section}
\numberwithin{theorem}{section}

\newcommand{\Z}{\mathbb{Z}}
\newcommand{\N}{\mathbb{N}}
\newcommand{\R}{\mathbb{R}}

\newcommand{\E}{\mathbb{E}}
\renewcommand{\P}{\mathbb{P}}
\newcommand{\F}{\mathcal{F}}
\newcommand{\Zd}{\mathbb{Z}^d}
\newcommand{\Rd}{\mathbb{R}^d}
\newcommand{\Sd}{\mathbb{S}^d}
\newcommand{\ep}{\varepsilon}

\newcommand{\ve}{\varepsilon}
\newcommand{\la}{\lambda}
\newcommand{\La}{\Lambda}
\newcommand{\vp}{\varphi}

\renewcommand{\L}{\mathcal{L}}
\renewcommand{\fint}{\strokedint}
\newcommand{\im}{m}

\newcommand{\V}{\mathbb{V}}
\newcommand{\err}{\mathcal{E}}

\DeclareMathOperator{\dist}{dist}

\DeclareMathOperator*{\osc}{osc}

\DeclareMathOperator{\var}{var}

\DeclareMathOperator{\tr}{tr}

\DeclareMathOperator{\divg}{div}

\newcommand{\X}{\mathcal{X}}  
\newcommand{\Y}{\mathcal{Y}}

\renewcommand{\tilde}{\widetilde}
\renewcommand{\div}{\divg}

\renewcommand{\hat}{\widehat}

\begin{document}

\title[Optimal quantitative estimates in stochastic homogenization]{Optimal quantitative estimates in stochastic homogenization for elliptic equations in nondivergence form}

\begin{abstract}
We prove quantitative estimates for the stochastic homogenization of linear uniformly elliptic equations in nondivergence form. Under strong independence assumptions on the coefficients, we obtain optimal estimates on the subquadratic growth of the correctors with stretched exponential-type bounds in probability. Like the theory of Gloria and Otto~\cite{GO1,GO2} for divergence form equations, the arguments rely on nonlinear concentration inequalities combined with certain estimates on the Green's functions and derivative bounds on the correctors. We obtain these analytic estimates by developing a~$C^{1,1}$ regularity theory down to microscopic scale, which is of independent interest and is inspired by the~$C^{0,1}$ theory introduced in the divergence form case by the first author and Smart~\cite{AS2}.
\end{abstract}

\author[S. Armstrong]{Scott Armstrong}
\address[S. Armstrong]{Universit\'e Paris-Dauphine, PSL Research University, CNRS, UMR [7534], CEREMADE, Paris, France}
\curraddr{Courant Institute of Mathematical Sciences, New York University, 251 Mercer St., New York 10012}
\email{scotta@cims.nyu.edu}

\author[J. Lin]{Jessica Lin}
\address[J. Lin]{Department of Mathematics, University of Wisconsin, Madison}
\email{jessica@math.wisc.edu}

\keywords{stochastic homogenization, correctors, error estimate}
\subjclass[2010]{35B27, 35B45, 35J15}
\date{\today}

\maketitle


\section{Introduction} 
\label{s.introduction}

\subsection{Motivation and informal summary of results}
We identify the \emph{optimal} error estimates for the stochastic homogenization of solutions $u^{\ep}$ solving:
\begin{equation}\label{e.maineq}
\begin{cases}
-\tr \left(A\left(\frac{x}{\ep}\right)D^{2}u^{\ep}\right)=0&\mbox{in}\quad U,\\
u^{\ep}(x)=g(x)&\mbox{on}\quad \partial U.
\end{cases}
\end{equation}
Here $U$ is a smooth bounded subset of $\Rd$ with $d\geq 2$, $D^2v$ is the Hessian of a function $v$ and $\tr(M)$ denotes the trace of a symmetric matrix $M\in \Sd$. The coefficient matrix~$A(\cdot)$ is assumed to be a stationary random field, with given law~$\P$, and valued in the subset of symmetric matrices with eigenvalues belonging to the interval~$[\lambda,\Lambda]$ for given ellipticity constants $0<\lambda \leq \Lambda$. The solutions $u^{\ep}$ are understood in the \emph{viscosity sense}~\cite{CIL} although in most of the paper the equations can be interpreted classically. We assume that the probability measure $\P$ has a product-type structure and in particular possesses a \emph{finite range of dependence} (see Section \ref{ss.assump} for the precise statement). According to the general qualitative theory of stochastic homogenization developed in~\cite{PV2,Y0} for nondivergence form elliptic equations (see also the later work~\cite{CSW}), the solutions $u^{\ve}$ of \eqref{e.maineq} converge uniformly as $\ep \to 0$, $\P$--almost surely, to that of the homogenized problem 
\begin{equation}
\begin{cases}
-\tr(\overline{A} D^{2}u)=0&\mbox{in}\quad U,\\
u(x)=g(x)&\mbox{on}\quad \partial U,
\end{cases}
\end{equation}
for some deterministic, uniformly elliptic matrix $\overline{A}$.  Our interest in this paper is to study the rate of convergence of $u^\ep$ to $u$. 

\smallskip

Error estimates quantifying the speed homogenization of $u^{\ve}\rightarrow u$ have been obtained in~\cite{Y0, Y2,CS, AS3}. The most recent paper~\cite{AS3} was the first to give a general result stating that the typical size of the error is at most algebraic, that is, $O(\ep^{\alpha})$ for some positive exponent $\alpha$. The earlier work~\cite{Y2} gave an algebraic error estimate in dimensions $d>4$. The main purpose of this paper is to reveal explicitly the optimal exponent. 

\smallskip

Our main quantitative estimates concern the size of certain stationary solutions called the \emph{approximate correctors}. These are defined, for a fixed symmetric matrix $M\in\mathbb{S}^{d}$ and $\ep>0$, as the unique solution $\phi_\ep \in C(\Rd) \cap L^\infty(\Rd)$ of the equation
\begin{equation}\label{e.approxcorrector}
\ep^{2}\phi_{\ep}-\tr\left(A(x)(M+D^{2}\phi_{\ep})\right)=0\quad\mbox{in} \ \Rd.
\end{equation}
Our main result states roughly that, for every $x\in\Rd$, $\ep \in (0,\tfrac12]$ and $t >0$,
\begin{equation} \label{e.opt}
\P \Big[ \big| \ep^2\phi_\ep(x) - \tr\left(\overline{A}M \right) \big| \geq t \err(\ep) \Big] \lesssim \exp \left( -t^{\frac12} \right),
\end{equation}
where the typical size $\err(\ep)$ of the error depends only on the dimension~$d$ in the following way:
\begin{equation} \label{e.error}
\err(\ep):= \left\{ \begin{aligned}
& \ep \left| \log \ep \right| && \mbox{in} \ d = 2, \\
& \ep^{\frac32} && \mbox{in} \ d=3,\\ 
& \ep^2 \left| \log \ep \right|^{\frac12} && \mbox{in} \ d=4,\\ 
& \ep^2  & & \mbox{in} \ d > 4.
\end{aligned} \right.
\end{equation}
Note that the rescaling $\phi^\ep(x):= \ep^2 \phi_\ep\left( \tfrac x\ep \right)$ allows us to write~\eqref{e.approxcorrector} in the so-called theatrical scaling as
\begin{equation*}
\phi^{\ep}-\tr\left(A\left( \tfrac x\ep\right) \left(M+D^{2}\phi^{\ep}\right)\right)=0
\quad\mbox{in} \ \Rd.
\end{equation*}
This is a well-posed problem (it has a unique bounded solution) on $\Rd$ which homogenizes to the equation
\begin{equation*}
\phi -\tr\left(\overline{A} \left(M+D^{2}\phi\right)\right)=0
\quad\mbox{in} \ \Rd.
\end{equation*}
The solution of the latter is obviously the constant function $\phi\equiv \tr\left( \overline{A}M \right)$ and so the limit 
\begin{equation}
\label{e.quallimit}
\ep^2\phi_\ep(x) = \phi^\ep(\ep x) \rightarrow  \tr\left( \overline{A}M \right)
\end{equation}
is a qualitative homogenization statement. Therefore, the estimate~\eqref{e.opt} is a quantitative homogenization result for this particular problem which asserts that the speed of homogenization is $O(\err(\ep))$. 

\smallskip

Moreover, it is well-known that estimating the speed of homogenization for the Dirichlet problem is essentially equivalent to obtaining estimates on the approximate correctors (see~\cite{Evans,AL2,CS, AS3}). Indeed, the estimate~\eqref{e.opt} can be transferred without any loss of exponent to an estimate on the speed of homogenization of the Dirichlet problem. One can see this from the standard two-scale ansatz
\begin{equation*}
u^\ep (x) \approx u(x) + \ep^2\phi_\ep\left( \tfrac x\ep \right),
\end{equation*}
which is easy to formalize and quantify in the linear setting since the homogenized solution~$u$ is completely smooth. We remark that since~\eqref{e.opt} is an estimate at a single point $x$, an estimate in $L^\infty$ for the Dirichlet problem will necessarily have an additional logarithmic factor $| \log \ep|^q$ for some $q(d)<\infty$. Since the argument is completely deterministic and essentially the same as in the case of periodic coefficients, we do not give the details here and instead focus on the proof of~\eqref{e.opt}.

\smallskip

The estimate~\eqref{e.opt} can also be expressed in terms of an estimate on the subquadratic growth of the correctors $\phi$, which are the solutions, for given $M\in\Sd$, of the problem
\begin{equation}\label{e.corrector}
\left\{ 
\begin{aligned}
& -\tr\left(A(x)(M+D^{2}\phi \right)= -\tr(\overline{A}M) & \mbox{in} & \ \Rd, \\
& \limsup_{R\to \infty} \, R^{-2} \osc_{B_R} \phi=0.\\
\end{aligned}
\right.
\end{equation}
Recall that while $D^2\phi$ exists as a stationary random field, $\phi$ itself may not be stationary. The estimate~\eqref{e.opt} implies that, for every $x\in\Rd$, 
\begin{equation*}
\P \left[ \left| \phi(x) - \phi(0) \right| \geq t |x|^2 \err\left(|x|^{-1}\right) \right] \lesssim \exp\left( -t^{\frac12} \right).
\end{equation*}
Notice that in dimensions $d>4$, this implies that the typical size of $|\phi(x) - \phi(0)|$ stays bounded as $|x| \to \infty$, suggesting that $\phi$ is a locally bounded, stationary random field. In Section~\ref{s.correctors}, we prove that this is so: the correctors are locally bounded and stationary in dimensions~$d>4$.

\smallskip

The above estimates are optimal in the size of the scaling, that is, the function $\err(\ep)$ cannot be improved in any dimension. This can be observed by considering the simple operator $-a(x)\Delta$ where $a(x)$ is a scalar field with a random checkerboard structure. Fix a smooth (deterministic) function $f\in C^\infty_c(\Rd)$ and consider the problem
\begin{equation}
\label{e.easyexample}
-a\left(\tfrac x\ep \right) \Delta u^\ep = f(x) \quad \mbox{in} \ \Rd.
\end{equation}
We expect this to homogenize to a problem of the form
\begin{equation*}
-\overline a \Delta u = f(x) \quad \mbox{in} \ \Rd.
\end{equation*}
In dimension $d=2$ (or in higher dimensions, if we wish) we can also consider the Dirichlet problem with zero boundary conditions in a ball much larger than the support of $f$. We can then move the randomness to the right side of the equation to have
\begin{equation*}
-\Delta u^\ep = a^{-1}\left(\tfrac x\ep \right) f(x) \quad \mbox{in} \ \Rd
\end{equation*}
and then write a formula for the value of the solution $u^\ep$ at the origin as a convolution of the random right side against the Green's kernel for the Laplacian operator. The size of the fluctuations of this convolution is easy to determine, since it is essentially a sum (actually a convolution) of i.i.d.~random variables, and it turns out to be precisely of order $\err(\ep)$ (with a prefactor constant depending on the variance of $a(\cdot)$ itself). For instance, we show roughly that
\begin{equation*}
\var\left[ u^\ep(0) \right] \simeq \err(\ep)^2.
\end{equation*}
This computation also serves as a motivation for our proof of~\eqref{e.opt}, although handling the general case of a random diffusion matrix is of course much more difficult than that of~\eqref{e.easyexample}, in which the randomness is scalar and can be split from the operator.

\subsection{Method of proof and comparison to previous works}

The arguments in this paper are inspired by the methods introduced in the divergence form setting by Gloria and Otto~\cite{GO1,GO2} and Gloria, Neukamm and Otto~\cite{GNO1} (see also Mourrat~\cite{M}). The authors combined certain concentration inequalities and analytic arguments to prove optimal quantitative estimates in stochastic homogenization for linear elliptic equations of the form
\begin{equation*}
-\nabla \cdot \left( A(x) \nabla u \right) = 0.
\end{equation*}
The concentration inequalities provided a convenient mechanism for transferring quantitative ergodic information from the coefficient field to the solutions themselves, an idea which goes back to an unpublished paper of Naddaf and Spencer~\cite{NS}. Most of these works rely on some version of the Efron-Stein inequality~\cite{ES} or the logarithmic Sobolev inequality to control the fluctuations of the solution by estimates on the spatial derivatives of the Green's functions and the solution.

\smallskip

A variant of these concentration inequalities plays an analogous role in this paper (see Proposition~\ref{p.concentration}). There are then two main analytic ingredients we need to conclude: first, an estimate on the decay of the Green's function for the heterogenous operator (note that, in contrast to the divergence form case, there is no useful deterministic bound on the decay of the Green's function); and (ii) a higher-order regularity theory asserting that, with high $\P$-probability, solutions of our random equation are more regular than the deterministic regularity theory would predict. We prove each of these estimates by using the sub-optimal (but algebraic) quantitative homogenization result of~\cite{AS3}: we show that, since solutions are close to those of the homogenized equation on large scales, we may ``borrow" the estimates from the constant-coefficient equation. This is an idea that was introduced in the context of stochastic homogenization for divergence form equations by the first author and Smart~\cite{AS2} (see also~\cite{GNO2,AM}) and goes back to work of Avellaneda and Lin~\cite{AL1,AL2} in the case of periodic coefficients. 

\smallskip

We remark that, while the scaling of the error $\err(\ep)$ is optimal, the estimate~\eqref{e.opt} is almost certainly sub-optimal in terms of stochastic integrability. This seems to be one limitation of an approach relying on (nonlinear) concentration inequalities, which so far yield only estimates with exponential moments~\cite{GNO2} rather than Gaussian moments~\cite{AKM1,AKM2,GO4}. Recently, new approaches based on renormalization-type arguments (rather than nonlinear concentration inequalities) have been introduced in the divergence form setting~\cite{AKM1,AKM2,GO4}. It was shown in~\cite{AKM2} that this approach yields estimates at the critical scale which are also optimal in stochastic integrability. It would be very interesting to see whether such arguments could be developed in the nondivergence form case.

\subsection{Assumptions}
\label{ss.assump}
In this subsection, we introduce some notation and present the hypotheses. Throughout the paper, we fix ellipticity constants $0< \lambda\leq \Lambda$ and the dimension $d \geq 2$. The set of $d$-by-$d$ matrices is denoted by $\mathbb{M}^{d}$, the set of $d$-by-$d$ symmetric matrices is denoted by $\mathbb{S}^{d}$, and $Id \in \mathbb{S}^d$ is the identity matrix. If $M,N \in \mathbb{S}^d$, we write $M \leq N$ if every eigenvalue of $N-M$ is nonnegative; $|M|$ denotes the largest eigenvalue of $M$. 

\subsubsection{Definition of the probability space $(\Omega,\F)$}
We begin by giving the structural hypotheses on the equation. 
We consider matrices $A\in \mathbb{M}^{d}$
which are uniformly elliptic:
\begin{equation} \label{e.Fellip}
\lambda Id\leq A(\cdot) \leq \Lambda Id
\end{equation}
and H\"older continuous in $x$: 
\begin{equation} \label{e.Fcont}
\exists\sigma \in (0,1] \quad \mbox{such that} \quad \sup_{x,y\in\Rd} \frac{\left| A(x)-A(y) \right|}{|x-y|^\sigma} < \infty,
\end{equation}
We define the set 
\begin{equation} \label{e.Omega}
\Omega:= \left\{ A \, :\,  \mbox{$A$ satisfies~\eqref{e.Fellip} and~\eqref{e.Fcont}} \right\}.
\end{equation}
Notice that the assumption~\eqref{e.Fcont} is a qualitative one. We purposefully do not specify any quantitative information regarding the size of the supremum in~\eqref{e.Fcont}, because none of our estimates depend on this value. We make this assumption in order to ensure that a comparison principle holds for our equations. 

\smallskip

We next introduce some $\sigma$--algebras on $\Omega$. For every Borel subset $U\subseteq \Rd$ we define $\F(U)$ to be the $\sigma$--algebra on $\Omega$ generated by the behavior of $A$ in $U$, that is,
\begin{multline} \label{e.FU}
\F(U):= \mbox{$\sigma$--algebra on $\Omega$ generated by the family of random variables} \\ \left\{ A\mapsto A(x) \,:\, \ x\in U \right\}.
\end{multline}
We define $\F := \F(\Rd)$.

\subsubsection{Translation action on $(\Omega,\F)$}

The translation group action on $\Rd$ naturally pushes forward to $\Omega$. We denote this action by $\{ T_y \}_{y\in \Rd}$, with $T_y:\Omega \to \Omega$ given by $$(T_yA)(x): =A(x+y).$$
The map $T_y$ is clearly $\F$--measurable, and is extended to~$\F$ by setting, for each $E \in \F$,~$$T_yE:= \left\{ T_yA \,:\, A\in E\right\}.$$ We also extend the translation action to $\F$--measurable random elements $X:\Omega\to S$ on $\Omega$, with $S$ an arbitrary set, by defining $(T_yX)(F):=X(T_yF)$.

\smallskip

We say that a random field $f:\Zd \times \Omega \to S$ is \emph{$\Zd$--stationary} provided that $f(y+z,A) = f(y,T_zA)$ for every $y,z\in\Zd$ and $A\in \Omega$. Note that an $\F$--measurable random element $X:\Omega\to S$ may be viewed as a $\Zd$--stationary random field via the identification with $\widetilde X(z,A):= X(T_zA)$.

\subsubsection{Assumptions on the random environment}
Throughout the paper, we fix $\ell \geq 2\sqrt{d}$ and a probability measure $\P$ on $(\Omega,\F)$ which satisfies the following:
\begin{itemize}

\item[(P1)] $\P$ has $\Zd$--stationary statistics: that is, for every $z\in\Zd$ and $E\in \F$,
\begin{equation*}
 \P \left[ E \right] = \P \left[ T_z E\right]. 
\end{equation*}

\item[(P2)] $\P$ has a range of dependence at most $\ell$: that is, for all Borel subsets $U,V$ of $\Rd$ such that $\dist(U,V) \geq \ell$, 
\begin{equation*} 
\mbox{ $\F(U)$ and $\F(V)$ \ are \ $\P$--independent.}
\end{equation*}
\end{itemize}

Some of our main results rely on concentration inequalities (stated in Section~\ref{ss.concentration}) which require a stronger independence assumption than finite range of dependence (P2) which was used in \cite{AS3}. Namely, we require that~$\P$ is the pushforward of another probability measure which has a product space structure. We consider a probability space $(\Omega_0,\F_0,\P_0)$ and denote
\begin{equation} \label{e.starform}
(\Omega_*,\F_*,\P_*) := (\Omega_0^{\Zd},\F_0^{\Zd},\P_0^{\Zd}).
\end{equation}
We regard an element $\omega\in \Omega_*$ as a map $\omega:\Zd \to \Omega_0$. For each $\Gamma \subseteq\Zd$, we denote by $\F_*(\Gamma)$ the $\sigma$-algebra generated by the family $\{ \omega\mapsto \omega(z) \,:\, z\in \Gamma\}$ of maps from $\Omega_*$ to $\Omega_0$. We denote the expectation with respect to $\P_*$ by $\E_*$. Abusing notation slightly, we also denote the natural $\Zd$-translation action on $\Omega_*$ by $T_z$, that is, $T_z:\Omega_*\to \Omega_*$ is defined by $(T_z\omega)(y):=\omega(y+z)$.

\smallskip

We assume that there exists an $(\F_*,\F)$--measurable map $\pi:\Omega_* \to \Omega$ which satisfies the following:
\begin{equation} \label{e.pullback}
\P \left[ E \right] = \P_* \!\left[ \pi^{-1} (E)\right] \quad \mbox{for every $E \in \F$,}  
\end{equation}
\begin{equation} \label{e.transcommute}
\pi \circ T_z = T_z \circ \pi \quad \mbox{for every} \ z\in \Zd,
\end{equation}
with the translation operator interpreted on each side in the obvious way, and
\begin{multline} \label{e.siginclusion}
\mbox{for every Borel subset $U \subseteq \Rd$ and $E \in \F(U)$,}     \\ \pi^{-1}(E) \in \F_{*}\left( \left\{ z\in \Zd \,:\, \dist(z,U) \leq \ell/2 \right\} \right).
\end{multline}
We summarize the above conditions as:
\begin{itemize}
\item[(P3)] There exists a probability space $(\Omega_*,\F_*,\P_*)$ of the form~\eqref{e.starform} and a map $$\pi:=\Omega_* \to \Omega,$$ which is $(\F,\F_*)$--measurable and satisfies~\eqref{e.pullback},~\eqref{e.transcommute} and~\eqref{e.siginclusion}.

\end{itemize}

Note that, in view of the product structure, the conditions~\eqref{e.pullback} and~\eqref{e.siginclusion} imply~(P2) and~\eqref{e.pullback} and~\eqref{e.transcommute} imply~(P1). Thus~(P1) and~(P2) are superseded by~(P3).

\subsection{Statement of main result}

We next present the main result concerning quantitative estimates of the approximate correctors. 

\begin{theorem}
\label{t.correctors}
Assume that $\P$ is a probability measure on $(\Omega,\F)$ satisfying~(P3). Let $\mathcal{E}(\ep)$ be defined by~\eqref{e.error}. Then there exist $\delta(d,\lambda,\Lambda)>0$ and $C(d,\lambda,\Lambda,\ell) \geq 1$ such that, for every $\ep \in (0,\tfrac12]$, 
\begin{equation} \label{e.correctorerror}
\E \left[  \exp \left(\left( \frac1{\err(\ep)} \sup_{x\in B_{\sqrt{d}}(0)}\left| \ep^2 \phi_\ep(x) - \tr\left( \overline{A}M \right) \right| \right)^{\frac{1}{2} + \delta } \,\right) \right] \leq C.
\end{equation}
\end{theorem}
The proof of Theorem~\ref{t.correctors} is completed in Section~\ref{s.optimals}.

\subsection{Outline of the paper} 
The rest of the paper is organized as follows. In Section~\ref{s.prelims}, we introduce the approximate correctors and the modified Green's functions and give some preliminary results which are needed for our main arguments. In Section~\ref{s.reg}, we establish a~$C^{1,1}$ regularity theory down to unit scale for solutions. Section~\ref{s.green} contains estimates on the modified Green's functions, which roughly state that the these functions have the same rate of decay at infinity as the Green's function for the homogenized equation (i.e, the Laplacian). In this section, we also mention estimates on the invariant measure associated to the linear operator in~\eqref{e.maineq}. In Section~\ref{s.sensitivity}, we use results from the previous sections to measure the ``sensitivity'' of solutions of the approximate corrector equation with respect to the coefficients. Finally, in Section~\ref{s.optimals}, we obtain the optimal rates of decay for the approximate corrector, proving our main result, and, in Section~\ref{s.correctors}, demonstrate the existence of stationary correctors in dimensions~$d>4$. In the appendix, we give a proof of the concentration inequality we use in our analysis, which is a stretched exponential version of the Efron-Stein inequality.

\section{Preliminaries}
\label{s.prelims}

\subsection{Approximate correctors and modified Green's functions}

For each given~$M\in\mathbb{S}^{d}$ and $\ep>0$, the approximate corrector equation is
\begin{equation*}
\ep^{2}\phi_{\ep}-\tr(A(x)(M+D^{2}\phi_{\ep}))=0\quad\mbox{in} \ \Rd.
\end{equation*}
The existence of a unique solution $\phi_\ep$ belonging to $C(\Rd)\cap L^\infty(\Rd)$ is obtained from the usual Perron method and the comparison principle. 

\smallskip

We also introduce, for each $\ep\in (0,1]$ and $y\in\Rd$, the ``modified Green's function" $G_{\ep}(\cdot, y; A)=G_\ep(\cdot,y)$, which is the unique solution of the equation
\begin{equation} \label{e.GF}
\ep^2 G_\ep -\tr\left(A (x) D^2G_\ep \right) = \chi_{B_\ell(y)} \quad \mbox{in} \ \Rd.
\end{equation}
Compared to the usual Green's function, we have smeared out the singularity and added the zeroth-order ``massive" term. 

To see that~\eqref{e.GF} is well-posed, we build the solution $G_{\ep}$ by compactly supported approximations. We first solve the equation in the ball $B_R$ (for $R>\ell$) with zero Dirichlet boundary data to obtain a function $G_{\ep,R}(\cdot,y)$. By the maximum principle, $G_{\ep,R} (\cdot,y)$ is increasing in $R$ and we may therefore define its limit as $R\to \infty$ to be $G_{\ep}(\cdot,y)$. We show in the following lemma that it is bounded and decays at infinity, and from this it is immediate that it satisfies~\eqref{e.GF} in~$\Rd$. The lemma is a simple deterministic estimate which is useful only in the regime $|x-y| \gtrsim \ep^{-1} \left| \log \ep \right|$ and follows from the fact (as demonstrated by a simple test function) that the interaction between the terms on the left of~\eqref{e.GF} give the equation a characteristic length scale of~$\ep^{-1}$. 

\begin{lemma}
\label{l.Gtails}
There exist $a(d,\lambda,\Lambda)>0$ and $C(d,\lambda,\Lambda)>1$ such that, for every $A\in \Omega$, $x,y\in \Rd$ and $\ep\in(0,1]$,
\begin{equation} \label{e.Gtails}
G_\ep(x,y) \leq C \ep^{-2} \exp\left( -\ep a |x-y|\right).
\end{equation}
\end{lemma}
\begin{proof}
Without loss of generality, let $y=0$. Let $\phi(x):=  C\exp\left( -\ep a |x|   \right)$ for $C,a>0$ to be selected below. Compute, for every $x\neq 0$, 
\begin{equation*} \label{}
D^2\phi(x) = \phi(x) \left(  -\ep a \frac{1}{|x|} \left( I - \frac{x\otimes x}{|x|^2}\right) + \ep^2 a^2 \frac{x\otimes x}{|x|^2} \right).
\end{equation*}
Thus for any $A\in\Omega$,
\begin{equation*} \label{}
 - \tr\left( A(x) D^2 \phi(x) \right) \geq \phi(x) \left( \frac{1}{|x|}\lambda\ep  a (d-1) - \Lambda(\ep^2 a^2) \right) \quad \mbox{in} \ \Rd\setminus \{ 0 \}. 
\end{equation*}
Set $a:=\frac{1}{\sqrt{2\La}}$. Then 
\begin{equation} \label{e.detsupersol}
\ep^{2}\phi - \tr\left( A(x) D^2 \phi(x) \right) \geq \ep^{2}\phi(x)(1-\La a^{2})+\frac{\phi(x)}{|x|}\la\ep a(d-1)\geq 0\quad \mbox{in} \ \Rd\setminus \{ 0 \}.
\end{equation}
Take $C := \exp(a\ell)$ so that $\phi(x) \geq 1$ in $|x| \leq \ell$. Define $\tilde \phi: = \ep^{-2} \min\{ 1, \phi \}$. Then $\tilde \phi$ satisfies the inequality
\begin{equation*} \label{}
\ep^2 \phi(x) - \tr\left( A(x) D^2 \phi(x) \right) \geq \chi_{B_{\ell}} \quad \mbox{in} \ \Rd.  
\end{equation*}
As $\tilde \phi>0$ on $\partial  B_R$, the comparison principle yields that $\tilde \phi \geq G_{\ep,R}(\cdot,0)$ for every $R>1$. Letting $R\to \infty$ yields the lemma.
\end{proof}

\subsection{Spectral gap inequalities}\label{ss.concentration}

In this subsection, we state the probabilistic tool we use to obtain the quantitative estimates for the modified correctors. The result here is applied precisely once in the paper, in Section~\ref{s.optimals}, and relies on the stronger independence condition (P3). It is a variation of the Efron-Stein (``spectral gap") inequality; a proof is given in Appendix~\ref{s.spectralgap}.

\begin{proposition}
\label{p.concentration}
Fix $\beta \in (0,2)$. Let $X$ be a random variable on $(\Omega_*,\F_*, \P_{*})$ and set 
\begin{equation*} \label{}
X_z' := \E_* \left[ X \,\vert\, \F_*(\Zd\setminus \{ z\}) \right] \qquad \mbox{and} \qquad \V_*\!\left[ X \right] := \sum_{z\in\mathbb{Z}^{d}} (X-X'_z)^2.
\end{equation*}
Then there exists $C(\beta)\geq 1$ such that
\begin{equation} \label{e.concentration}
\E_* \!\left[ \exp\left( \left| X - \E \left[ X\right] \right|^\beta \right) \right] \leq C \E_* \!\left[ \exp\left( \left(C\V_*\!\left[ X \right] \right)^{\frac{\beta}{2-\beta}} \right) \right]^\frac{2-\beta}{\beta}.
\end{equation}
\end{proposition}

The conditional expectation $X_z'$ can be identified by resampling the random environment near the point $z$ (this is explained in depth in Section \ref{s.sensitivity}). Therefore, the quantity $(X-X'_{z})$ measures changes in $X$ with respect to changes in the environment near $z$. Following~\cite{GNO1}, we refer to $(X-X'_{z})$ as the \emph{vertical derivative} of $X$ at the point $z$.

\subsection{Suboptimal error estimate}

We recall the main result of \cite{AS3}, which is the basis for much of the analysis in this paper. We reformulate their result slightly, to put it in a form which is convenient for our use here.

\begin{proposition}
\label{p.subopt}
Fix $\sigma \in (0,1]$ and $s \in (0,d)$. Let $\mathbb{P}$ satisfy (P1) and (P2). There exists an exponent $\alpha(\sigma,s,d,\lambda,\Lambda,\ell)>0$ and a nonnegative random variable $\X$ on $(\Omega,\F)$, satisfying
\begin{equation} \label{e.expboundX}
\E \left[ \exp\left( \X \right) \right]  \leq C(\sigma,s,d,\Lambda,\lambda,\ell) < \infty
\end{equation}
and such that the following holds: for every $R\geq1$, $f\in C^{0,\sigma}(B_R)$, $g\in C^{0,\sigma}(\partial B_R)$ and solutions $u,v\in C(\overline B_R)$ of the Dirichlet problems
\begin{equation} \label{e.DP}
\left\{ 
\begin{aligned}
& -\tr \left(A(x)D^{2}u\right) = f(x)& \mbox{in} & \ B_{R}, \\
& u = g & \mbox{on} & \ \partial B_R,
\end{aligned} 
\right.
\end{equation}
and
\begin{equation}
\label{e.DP2}
\left\{ 
\begin{aligned}
& -\tr \left(\overline{A}D^{2}v\right) = f(x) & \mbox{in} & \ B_{R}, \\
& v = g & \mbox{on} & \ \partial B_R,
\end{aligned} 
\right.
\end{equation}
we have, for a constant $C(\sigma,s,d,\Lambda,\lambda,\ell)\geq1$, the estimate
\begin{equation*} \label{}
R^{-2} \sup_{B_R} \left| u(x) - v(x) \right| \leq C R^{-\alpha} \left(1+\X R^{-s} \right) \Gamma_{R,\sigma}(f,g), 
\end{equation*}
where
\begin{equation*} \label{}
\Gamma_{R,\sigma}(f,g) := \sup_{B_R} \left| f \right| + R^{\sigma} \left[ f \right]_{C^{0,\sigma}(B_R)}  +  R^{-2} \osc_{\partial B_R} g + R^{-2+\sigma} \left[ g \right]_{C^{0,\sigma}(\partial B_R)}.
\end{equation*}
\end{proposition}

We note that, in view of the comparison principle, Proposition~\ref{p.subopt} also give one-sided estimates for subsolutions and supersolutions.

\section{Uniform \texorpdfstring{$C^{1,1}$}{C11} Estimates}\label{s.reg}

In this section, we present a large-scale ($R\gg 1$) regularity theory for solutions of the equation
\begin{equation}
\label{e.pder}
-\tr\left(A(x)D^{2}u\right)=f(x) \quad \mbox{in} \ B_R
\end{equation}
Recall that according to the Krylov-Safonov H\"{o}lder regularity theory \cite{CC}, there exists an exponent~$\sigma(d,\lambda,\Lambda)\in (0,1)$ such that $u$ belongs to $C^{0, \sigma}(B_{R/2})$ with the estimate
\begin{equation}\label{e.classicKS}
R^{\sigma-2}\left[u\right]_{C^{0, \sigma}({B_{R/2})}}\leq C\left(R^{-2} \osc_{B_{R}} u+ \left( \fint_{B_R} \left| f(x)\right|^d\,dx \right)^{\frac1d}\right).
\end{equation}
This is the best possible estimate, independent of the size of $R$, for solutions of general equations of the form~\eqref{e.pder}, even if the coefficients are smooth. What we show in this section is that, due to statistical effects, solutions of our random equation are typically much more regular, at least on scales larger than~$\ell$, the length scale of the correlations of the coefficient field. Indeed, we will show that solutions, with high probability, are essentially $C^{1,1}$ on scales larger than the unit scale. 

\smallskip

The arguments here are motivated by a similar $C^{0,1}$ regularity theory for equations in divergence form developed in~\cite{AS2} and should be seen as a nondivergence analogue of those estimates. In fact, the arguments here are almost identical to those of~\cite{AS2}. They can also be seen as generalizations to the random setting of the results of Avellaneda and Lin for periodic coefficients, who proved uniform $C^{0,1}$ estimates for divergence form equations~\cite{AL1} and then $C^{1,1}$ estimates for the nondivergence case~\cite{AL2}. Note that it is natural to expect estimates in nondivergence form to be one derivative better than those in divergence form (e.g., the Schauder estimates). Note that the ``$C^{1,1}$ estimate'' proved in~\cite{GNO2} has a different statement which involves correctors; the statement we prove here would be simply false for divergence form equations.

\smallskip

The rough idea, similar to the proof of the classical Schauder estimates, is that, due to homogenization, large-scale solutions of~\eqref{e.pder} can be well-approximated by those of the homogenized equation. Since the latter are harmonic, up to a change of variables, they possess good estimates. If the homogenization is fast enough (for this we need the results of~\cite{AS3}, namely Proposition~\ref{p.subopt}), then the better regularity of the homogenized equation is inherited by the heterogeneous equation. This is a quantitative version of the idea introduced in the context of periodic homogenization by Avellaneda and Lin \cite{AL1, AL2}.

\smallskip

Throughout this section, we let $\mathcal{Q}$ be the set of polynomials of degree at most two 
and let $\mathcal{L}$ denote the set of affine functions. For $\sigma \in (0,1]$ and $U\subseteq\Rd$, we denote the usual H\"older seminorm by $\left[ \cdot\right]_{C^{0,\sigma}(U)}$.

\begin{theorem}[{$C^{1,1}$ regularity}]
\label{t.regularity}
Assume (P1) and (P2). Fix $s\in (0,d)$ and $\sigma \in (0,1]$. There exists an $\F$--measurable random variable $\X$ and a constant $C(s,\sigma,d,\lambda,\Lambda,\ell)\geq 1$ satisfying
\begin{equation} \label{e.scrbound}
\E \left[ \exp\left( \X^s  \right) \right] \leq C < \infty
\end{equation}
such that the following holds: for every $M\in\mathbb{S}^{d}$, $R\geq 2\X$, $u\in C(B_R)$ and $f\in C^{0,\sigma}(B_R)$ satisfying 
\begin{equation*} \label{}
-\tr\left(A(x)(M+D^{2}u) \right) = f(x) \quad \mbox{in} \ B_R
\end{equation*}
and every $r\in \left[ \X, \frac12 R \right]$,
 we have the estimate
\begin{multline} \label{e.pwC11}
\frac{1}{r^2} \inf_{l\in\L} \sup_{B_r} |u-l|
\\
\leq  C\left( \left|f(0)+\tr(\overline{A}M)\right| + R^{\sigma}\! \left[ f \right]_{C^{0,\sigma}(B_R)} + \frac{1}{R^2} \inf_{l\in\L} \sup_{B_R} |u-l| \right).
\end{multline}
\end{theorem}

It is well-known that any~$L^\infty$ function which can be well-approximated on all scales by smooth functions must be smooth. The proof of Theorem~\ref{t.regularity} is based on a similar idea: any function~$u$ which can be well-approximated on all scales above a fixed scale $\X$, which is of the same order as the microscopic scale, by functions~$w$ enjoying an \emph{improvement of quadratic approximation} property must itself have this property. This is formalized in the next proposition, the statement and proof of which are similar to those of~\cite[Lemma 5.1]{AS3}.

\begin{proposition}
\label{p.quadapprox}

For each $r>0$ and $\theta \in (0,\tfrac12)$, let $\mathcal A(r,\theta)$ denote the subset of $L^\infty(B_r)$ consisting of $w$ which satisfy
\begin{equation*} \label{}
\frac{1}{(\theta r)^2} \inf_{q\in \mathcal Q} \sup_{x\in B_{\theta r}} \left| w(x) - q(x) \right| \leq \frac12 \left(  \frac1{r^2} \inf_{q\in \mathcal Q} \sup_{x\in B_{r}} \left| w(x) - q(x) \right| \right).
\end{equation*}
Assume that $\alpha, \gamma, K,L>0$, $1 \leq 4r_0 \leq R$ and $u\in L^\infty(B_R)$ are such that, for every $r \in [r_0,R/2]$, there exists $v\in \mathcal A(r,\theta)$ such that 
\begin{equation} \label{e.quappr}
\frac1{r^2}\sup_{x\in B_r}\left| u(x) -  v(x) \right| \leq r^{-\alpha} \left( K +  \frac1{r^2}\inf_{l\in\L} \sup_{x\in B_{2r}} \left| u(x) - l(x) \right| \right) + Lr^\gamma.
\end{equation}
Then there exist $\beta(\theta) \in (0,1]$ and $C(\alpha,\theta,\gamma) \geq 1$ such that, for every $s\in [r_0,R/2]$, 
\begin{equation} \label{e.C11lem2}
\frac{1}{s^2} \inf_{l\in\L} \sup_{x\in B_s} \left| u(x) - l(x) \right| \leq C\left(K+LR^\gamma+\frac{1}{R^2} \inf_{l\in\L} \sup_{x\in B_R} \left| u(x) - l(x) \right| \right).
\end{equation}
and
\begin{multline} \label{e.C11lem1}
\frac{1}{s^2} \inf_{q\in\mathcal Q} \sup_{x\in B_s} \left| u(x) - q(x) \right| 
\leq C\left( \frac{s}{R} \right)^{\beta} \left( LR^\gamma+ \frac{1}{R^2} \inf_{q\in\mathcal Q} \sup_{x\in B_R} \left| u(x) - q(x) \right|  \right) \\
+ Cs^{-\alpha} \left( K+LR^\gamma + \frac{1}{R^2} \inf_{l\in\L} \sup_{x\in B_R} \left| u(x) - l(x) \right| \right).
\end{multline}
\end{proposition}

\begin{proof}
Throughout the proof, we let $C$ denote a positive constant which depends only on $(\alpha,\theta,\gamma)$ and may vary in each occurrence. We may suppose without loss of generality that $\alpha \leq 1$ and that $\gamma \leq c$ so that $\theta^\gamma \geq \frac23$.

\smallskip

\emph{Step 1.} We set up the argument. We keep track of the two quantitites
\begin{equation*} \label{}
G(r):= \frac1{r^2} \inf_{q\in \mathcal Q} \sup_{x\in B_r} \left| u(x) - q(x) \right| \quad \mbox{and} \quad H(r):=  \frac1{r^2} \inf_{l\in \mathcal L} \sup_{x\in B_r} \left| u(x) - l(x) \right|.
\end{equation*}
By the hypotheses of the proposition and the triangle inequality, we obtain that, for every $r\in [r_0,R/2]$,
\begin{equation} \label{e.improvequad}
G(\theta r) \leq \frac12 G(r) +C r^{-\alpha} \left( K + H(2r) \right)+ Lr^\gamma.
\end{equation}
Denote $s_0:=R$ and $s_{j}:= \theta^{j-1} R/4$ and let $m\in\N$ be such that 
\begin{equation*} \label{}
s_m^{-\alpha} \leq \frac14  \leq s_{m+1}^{-\alpha}.
\end{equation*}
Denote $G_j:= G(s_j)$ and $H_j := H(s_j)$. Noting that $\theta \leq \frac12$, from~\eqref{e.improvequad} we get, for every $j\in\{ 1,\ldots,m-1\}$,
\begin{equation} \label{e.improvequad2}
 G_{j+1} \leq \frac12  G_j + C s_j^{-\alpha} \left( K + H_{j-1} \right) + Ls_{j}^\gamma.
\end{equation}
For each $j\in\{ 0,\ldots,m-1\}$, we may select $q_j\in \mathcal Q$ such that 
\begin{equation*} \label{}
G_j = \frac{1}{s_j^2} \sup_{x\in B_{s_j}} \left| u(x) - q_j(x) \right|.
\end{equation*}
We denote the Hessian matrix of $q_j$ by $Q_j$. The triangle inequality implies that 
\begin{equation} \label{e.GjHj}
G_j \leq H_j \leq G_j + \frac{1}{s_j^2}\sup_{x\in B_{s_j}} \frac12 x \cdot Q_jx = G_j + \frac12 |Q_j|,
\end{equation}
and 
\begin{align*}
\frac{1}{s_j^2} \sup_{x\in B_{s_{j+1}}} \left| q_{j+1}(x) - q_j(x) \right|  \leq G_j + \theta^2 G_{j+1}.
\end{align*}
The latter implies 
\begin{equation*} 
\left| Q_{j+1} - Q_j \right| \leq \frac{2}{s_{j+1}^2} \sup_{x\in B_{s_{j+1}}} \left| q_{j+1}(x) - q_j(x) \right| \leq  \frac{2}{\theta^{2}} G_j + 2G_{j+1}.
\end{equation*}
In particular, 
\begin{equation}\label{e.Qjdiff}
|Q_{j+1}| \leq |Q_j| + C \left( G_j + G_{j+1} \right)
\end{equation}
Similarly, the triangle inequality also gives 
\begin{equation} \label{e.Qjstup}
|Q_j| = \frac{2}{s_j^2} \inf_{l\in\L} \sup_{x\in B_{s_j}} \left| q_j (x) - l(x) \right| \leq 2G_j + 2H_j \leq 4H_j.
\end{equation}
Thus, combining \eqref{e.Qjdiff} and \eqref{e.Qjstup}, yields
\begin{equation} \label{e.Qkbound}
|Q_j| \leq |Q_0| + C\sum_{i=0}^j G_i \leq C \left( H_0 + \sum_{i=0}^j G_i \right).
\end{equation}
Next, combining~\eqref{e.improvequad2}~\eqref{e.GjHj} and~\eqref{e.Qkbound}, we obtain, for every $j\in \{0,\ldots,m-1\}$,
\begin{align} \label{e.Gjdiff}
 G_{j+1} & \leq \frac12  G_{j} + C s_j^{-\alpha} \left( K +  H_0 + \sum_{i=0}^j G_i \right) + Ls_j^\gamma.
\end{align}
The rest of the argument involves first iterating~\eqref{e.Gjdiff} to obtain bounds on $G_j$, which yield bounds on $|Q_j|$ by~\eqref{e.Qkbound}, and finally on $H_j$ by~\eqref{e.GjHj}. 

\smallskip

\emph{Step 2.} We show that, for every $j\in \{ 1,\ldots, m\}$, 
\begin{equation} \label{e.GjdiffmaxQj}
G_j \leq 2^{-j} G_0 + C s_j^{-\alpha} \left( K+H_0 \right) + CL \left( s_j^\gamma+ R^\gamma s_j^{-\alpha}  \right). 
\end{equation}
We argue by induction. Fix $A,B\geq 1$ (which are selected below) and suppose that $k\in \{ 0,\ldots,m-1\}$ is such that, for every $j \in\{ 0,\ldots,k\}$, 
\begin{equation*} \label{}
G_j \leq 2^{-j} G_0 + As_j^{-\alpha} \left( K+H_0 \right) + L\left( A  s_j^\gamma+ BR^\gamma s_j^{-\alpha}  \right).
\end{equation*}
Using~\eqref{e.Gjdiff} and this induction hypothesis (and that $G_0 \leq H_0$), we get
\begin{align*}
G_{k+1} & \leq \frac12  G_{k} + C s_k^{-\alpha} \left( K +  H_0 + \sum_{j=0}^k G_j \right) + Ls_{k}^\gamma \\
& \leq \frac12 \left( 2^{-k} G_0 + As_k^{-\alpha} \left( K+H_0 \right) + L \left( As_k^\gamma+ BR^\gamma s_k^{-\alpha}  \right) \right) + Ls_k^\gamma \\
& \qquad + Cs_k^{-\alpha} \left(K+H_0+\sum_{j=0}^k \left( 2^{-j} G_0 + As_j^{-\alpha} \left( K+H_0 \right) + L \left(A s_j^\gamma+ BR^\gamma s_j^{-\alpha}  \right) \right) \right) \\
& \leq 2^{-(k+1)} G_0 + s_{k+1}^{-\alpha} (K+H_0) \left(\frac12 A + CAs_k^{-\alpha} +C  \right) \\ 
& \qquad + Ls_{k+1}^\gamma \left(  \frac12\theta^{-\gamma}A+C \right)  + L R^\gamma s_{k+1}^{-\alpha}  \left(   \frac12 B + CA + CBs_k^{-\alpha} \right).
\end{align*}
Now suppose in addition that $k\leq n$ with $n$ such that $Cs_n^{-\alpha} \leq \frac14$. Then using this and $\theta^{\gamma} \geq \frac23$, we may select $A$ large enough so that 
\begin{equation*} \label{}
\frac12 A + C A s_k^{-\alpha} + C\leq \frac34 A + C \leq A \quad \mbox{and} \quad \frac12 \theta^{-\gamma} A + C \leq A,
\end{equation*}
and then select $B$ large enough so that 
\begin{equation*} \label{}
\frac12 B + CA + CBs_k^{-\alpha}  \leq B.
\end{equation*}
We obtain
\begin{equation*} \label{}
G_{k+1} \leq 2^{-(k+1)} G_0 + As_{k+1}^{-\alpha} \left( K+H_0 \right) + ALs_{k+1}^\gamma + BLR^\gamma s_{k+1}^{-\alpha}.
\end{equation*}
By induction, we now obtain~\eqref{e.GjdiffmaxQj} for every $j \leq n\leq m$. In addition, for every $j \in \{ n+1,\ldots,m\}$, we have that $1 \leq s_j / s_m \leq C$. This yields~\eqref{e.GjdiffmaxQj} for every $j\in \{ 0,\ldots,n\}$.

\smallskip
%

\smallskip

\emph{Step 3.} The bound on $H_j$ and conclusion.
By~\eqref{e.Qkbound},~\eqref{e.GjdiffmaxQj}, we have
\begin{align*} \label{}
|Q_j| &  \leq C\left(H_0 + \sum_{i=0}^j G_i\right)   \\
& \leq C\left(H_0 + \sum_{i=0}^j \left( 2^{-i} G_0 + Cs_i^{-\alpha} (K+H_0) + CL(s_i^\gamma+R^\gamma s_i^{-\alpha} \right) \right)\\
& \leq C\left(H_0 + G_0 + Cs_j^{-\alpha} (K+H_0) + C LR^\gamma\left(1+s_j^{-\alpha} \right) \right)\\
& \leq CH_0 + CKs_j^{-\alpha} + CLR^\gamma. 
\end{align*}
Here we also used that $s_j^{-\alpha} \leq s_m^{-\alpha} \leq C$. Using the previous inequality,~\eqref{e.GjHj} and~\eqref{e.GjdiffmaxQj}, we conclude that 
\begin{equation*} \label{}
H_j \leq G_j + \frac{1}{2}|Q_j| \leq C H_0 + CKs_j^{-\alpha}+ CLR^\gamma.
\end{equation*}
This is~\eqref{e.C11lem2}. Note that~\eqref{e.GjdiffmaxQj} already implies~\eqref{e.C11lem1} for $\beta:=(\log 2) / |\log \theta|$.
\end{proof}

Next, we recall that solutions of the homogenized equation (which are essentially harmonic functions) satisfy the ``improvement of quadratic approximation'' property.

\begin{lemma}
\label{l.checkdyad}
Let $r>0$. Assume that $\overline{A}$ satisfies~\eqref{e.Fellip} and let $w\in C(B_r)$ be a solution of 
\begin{equation} 
\label{e.Ahom}
-\tr \left(\overline{A}D^2w\right)= 0 \quad \mbox{in} \ B_r.
\end{equation}
There exists~$\theta(d,\lambda,\Lambda) \in (0,\tfrac12]$ such that, for every $r>0$,
\begin{equation*} \label{}
\frac{1}{(\theta r)^2} \inf_{q\in\mathcal Q} \sup_{x\in B_{\theta r}} \left| w(x) - q(x) \right| \leq \frac12 \left( \frac{1}{ r^2} \inf_{q\in\mathcal Q} \sup_{x\in B_{r}} \left| w(x) - q(x) \right| \right).
\end{equation*}

\end{lemma}
\begin{proof}
Since $\overline A$ is constant, the solutions of~\eqref{e.Ahom} are harmonic (up to a linear change of coordinates). Thus the result of this lemma is classical.
\end{proof}
 
Equipped with the above lemmas, we now give the proof of Theorem~\ref{t.regularity}.

\begin{proof}[{Proof of Theorem~\ref{t.regularity}}]
Fix $s\in (0,d)$. We denote by $C$ and $c$ positive constants depending only on $(s,\sigma,d,\lambda,\Lambda,\ell)$ which may vary in each occurrence. We proceed with the proof of~\eqref{e.pwC11}. Let $\mathcal Y$ be the random variable in the statement of Proposition~\ref{p.subopt}, with $\alpha$ the exponent there. Define $\X:= \Y^{1/s}$ and observe that $\X$ satisfies~\eqref{e.scrbound}. We take $\sigma$ to be the smallest of the following: the exponent in~\eqref{e.classicKS} and half the exponent $\alpha$ in Proposition~\ref{p.subopt}.

\smallskip

We may suppose without loss of generality that $-\tr\left(\overline{A}M\right) = f(0) = 0$.

\smallskip

\emph{Step 1.} We check that~$u$ satisfies the hypothesis of Proposition~\ref{p.quadapprox} with $r_0= C\X$. Fix $r\in [ C\X,R/2]$. We take $v,w\in C(B_{3r/4})$ to be the solutions of the problems 
\begin{equation*} \label{}
\left\{ 
\begin{aligned}
& -\tr\left(\overline{A} D^2 v\right)= f(x)& \mbox{in} & \ B_{3r/4}, \\
& v = u & \mbox{on} & \ \partial B_{3r/4},
\end{aligned} 
\right. \qquad \left\{ 
\begin{aligned}
& -\tr \left(\overline{A} D^2w\right) = 0 & \mbox{in} & \ B_{{3r/4}}, \\
& w = u & \mbox{on} & \ \partial B_{3r/4}.
\end{aligned} 
\right.
\end{equation*}
By the Alexandrov-Bakelman-Pucci estimate \cite{CC}, we have
\begin{equation} \label{e.tregabp}
\frac1{r^2}\sup_{B_{r/2}} | v-w | \leq C \left( \fint_{B_r} \left| f(x) \right|^d \, dx \right)^{\frac1d} \leq C r^{\sigma} \left[ f \right]_{C^{0,\sigma}(B_r)}.
\end{equation}
By the Krylov-Safanov H\"older estimate \eqref{e.classicKS},
\begin{multline} \label{e.tregKSapp}
r^{\sigma-2}\left[ u \right]_{C^{0,\sigma}(B_{3r/4}) } \leq C \left( \frac{1}{r^2} \osc_{B_r} u + \left( \fint_{B_r} \left| f(x) \right|^d \, dx \right)^{\frac1d} \right) \\
\leq C \left( \frac{1}{r^2} \osc_{B_r} u + r^{\sigma} \left[ f \right]_{C^{0,\sigma}(B_r)} \right).
\end{multline}
By the error estimate (Proposition~\ref{p.subopt}), we have 
\begin{multline*} \label{}
\frac1{r^2}\sup_{B_{r/2}} | u-v | \\
\leq Cr^{-\alpha} \left( 1 + \Y r^{-s} \right) \left( r^\sigma \left[ f \right]_{C^{0,\sigma}(B_r)} + \frac 1{r^2} \osc_{\partial B_{3r/4}} u  +  r^{\sigma-2}\left[ u \right]_{C^{0,\sigma}(B_{3r/4}) }\right).
\end{multline*}
Using the assumption $r^s \geq \X^s  = \Y$ and~\eqref{e.tregKSapp}, this gives 
\begin{equation} \label{e.tregee}
\frac1{r^2}\sup_{B_{r/2}} | u-v | \leq Cr^{-\alpha} \left(  r^\sigma \left[ f \right]_{C^{0,\sigma}(B_r)} + \frac 1{r^2} \osc_{B_r} u \right).
\end{equation}
Using~\eqref{e.tregabp} and~\eqref{e.tregee}, the triangle inequality, and the definition of $\sigma$, we get
\begin{equation*} \label{}
\frac1{r^2}\sup_{B_{r/2}} | u-w | \leq Cr^{-\alpha} \left(  \frac 1{r^2} \osc_{B_r} u \right) + Cr^\sigma \left[ f \right]_{C^{0,\sigma}(B_R)}.
\end{equation*}
By Lemma~\ref{l.checkdyad}, $w\in \mathcal A(r,\theta)$ for some $\theta\geq c$.

\smallskip

\emph{Step 2.} We apply Proposition~\ref{p.quadapprox} to obtain~\eqref{e.pwC11}. The proposition gives, for every $r\geq r_0 = C\X$,
\begin{equation*} \label{}
\frac{1}{r^2} \inf_{l\in\L} \sup_{x\in B_{r}} \left| u(x) - l(x) \right| \leq C\left(R^\sigma \left[ f \right]_{C^{0,\sigma}(B_R)} +\frac{1}{R^2} \inf_{l\in\L} \sup_{x\in B_r} \left| u(x) - l(x) \right| \right), 
\end{equation*}
which is \eqref{e.pwC11}. 
\end{proof}

It is convenient to restate the estimate~\eqref{e.pwC11} in terms of ``coarsened" seminorms. Recall that, for $\phi \in C^\infty(B_1)$,
\begin{equation*}
\begin{aligned} \label{}
\left|D\phi(x_0) \right| & \simeq \frac{1}{h} \osc_{B_h(x_0)} \phi(x),  \\
\left|D^2\phi(x_0) \right| & \simeq \frac{1}{h^2}\inf_{p\in \Rd} \osc_{B_h(x_0)} \left( \phi(x) - p\cdot x \right),
\end{aligned} \qquad \mbox{for} \quad 0<h \ll 1.
\end{equation*}
This motivates the following definitions: for $\alpha\in (0,1]$, $h\geq 0$, $U\subseteq \Rd$ and $x_0\in U$, we define the \emph{pointwise, coarsened $h$-scale $C^{0,1}_{h}(U)$ and $C^{1,1}_{h}(U)$ seminorms at $x_0$} by
\begin{equation*} \label{}
\left[ \phi \right]_{C^{0,\alpha}_h(x_0,U)} := \sup_{r > h} \frac1{r^\alpha} \osc_{B_r(x_0) \cap U}  \phi,
\end{equation*}
and
\begin{equation*} \label{}
\left[ \phi \right]_{C^{1,\alpha}_h(x_0,U)} := \sup_{r > h} \frac1{r^{1+\alpha}} \inf_{l\in \mathcal{L}} \osc_{B_r(x_0) \cap U} \left( \phi(x) - l(x) \right).
\end{equation*}
This allows us to write \eqref{e.pwC11} in the form
\begin{multline}\label{e.pwcC11}
\left[ u \right]_{C^{1,1}_1(0,B_{R/2})} \\
\leq  C\X^{2}\left( \left|f(0)+\tr(\overline{A}M)\right| + R^{\sigma}\! \left[ f \right]_{C^{0,\sigma}(B_R)} + \frac{1}{R^2} \inf_{l\in\L} \sup_{x\in B_R} |u-l| \right).
\end{multline}

As a simple corollary to Theorem \ref{t.regularity}, we also have $C^{0,1}_{1}$ bounds on $u$:
\begin{corollary}\label{c.C01reg}
Assume the hypotheses and notation of Theorem \ref{t.regularity}. Then,
\begin{equation} \label{e.pwC01}
\left[ u \right]_{C^{0,1}_1(0,B_{R/2})} \leq  \X \left( R \left|f(0)+\tr \overline{A}M\right| + R^{1+\sigma} \left[ f \right]_{C^{0,\sigma}(B_R)} + \frac1R\osc_{B_R}u  \right).
\end{equation}
\end{corollary}

The proof follows from a simple interpolation inequality, which controls the seminorm $\left[ \cdot \right]_{ C^{0,1}_{h}(B_R)}$ in terms of $\left[ \cdot \right]_{C^{1,1}_{h}(B_R)}$ and the oscillation in~$B_R$:
\begin{lemma}
\label{l.interp}
For any $R>0$ and $\phi \in C(B_R)$, we have
\begin{equation} \label{e.interp}
\left[ \phi \right]_{C^{0,1}_h(0,B_R)} \leq 14 \left(\left[ \phi \right]_{C^{1,1}_h(0,B_R)}\right)^{\frac12} \left(\osc_{B_R} \phi\right)^{\frac12}. 
\end{equation}
\end{lemma}
\begin{proof}

We must show that, for every $s \in [h,R]$, 
\begin{equation} \label{e.verinterp}
\frac1s \osc_{B_s} \phi \leq 14 \left(\left[ \phi \right]_{C^{1,1}_h(0,B_R)}\right)^{\frac12} \left(\osc_{B_R} \phi\right)^{\frac12}.
\end{equation}
Set $$K:= \left(\left[ \phi \right]_{C^{1,1}_h(0,B_R)}\right)^{-\frac12} \left(\osc_{B_R} \phi\right)^{\frac12}$$
and observe that, for every $s\in[ K , R]$, we have
\begin{equation} \label{e.easyste}
\frac1s \osc_{B_s} \phi \leq K^{-1} \osc_{B_R} \phi =\left( \left[ \phi \right]_{C^{1,1}_h(0,B_R)}\right)^{\frac12} \left(\osc_{B_R} \phi\right)^{\frac12}.
\end{equation}
Thus we need only check~\eqref{e.verinterp} for $s\in [h,K]$.

\smallskip

We next claim that, for every $s\in [h,R)$, 
\begin{equation} \label{e.iters}
\frac{2}{s} \osc_{B_{s/2}} \phi \leq 3s \left[ \phi \right]_{C^{1,1}_h(0,B_R)} + \frac{1}{s} \osc_{B_s} \phi. 
\end{equation}
Fix $s$ and select $p\in\Rd$ such that 
\begin{equation*} \label{}
\frac{1}{s^2} \osc_{ B_s} \left( \phi(y) - p\cdot y \right) \leq \left[ \phi \right]_{C^{1,1}_h(0,B_R)}.
\end{equation*}
Then
\begin{equation*} \label{}
|p| = \frac{1}{2s} \osc_{ B_s} \left( -p\cdot y \right) \leq \frac1{2s} \osc_{ B_s} \left( \phi(y) - p\cdot y \right) + \frac1{2s} \osc_{B_s} \phi.
\end{equation*}
Together these yield
\begin{align*}
\frac{2}{s} \osc_{B_{s/2}} \phi & \leq \frac{2}{s} \osc_{ B_{s/2}} \left( \phi(y) - p\cdot y \right) + \frac{2}{s} \osc_{B_{s/2}} \left( -p\cdot y \right) \\
& = \frac{2}{s} \osc_{ B_{s/2}} \left( \phi(y) - p\cdot y \right) + 2|p| \\
& \leq \frac{3}{s} \osc_{ B_{s}} \left( \phi(y) - p\cdot y \right) + \frac1{s} \osc_{B_s} \phi \\
& \leq 3s \left[ \phi \right]_{C^{1,1}_h(0,B_R)} + \frac{1}{s} \osc_{B_s} \phi.
\end{align*}
This is~\eqref{e.iters}.

We now iterate~\eqref{e.iters} to obtain the conclusion for $s\in [h,K]$. By induction, we see that for each $j\in\N$ with $R_j := 2^{-j} K \geq h$,
\begin{align*} \label{}
R_{j}^{-1} \osc_{B_{R_{j}}} \phi & \leq K^{-1} \osc_{B_K} \phi + 3 \left( \sum_{i=0}^{j-1} R_i \right)  \left[ \phi \right]_{C^{1,1}_h(0,B_R)} \\
& \leq K^{-1} \osc_{B_K} \phi + 6K \left[ \phi \right]_{C^{1,1}_h(0,B_R)}.
\end{align*}
Using~\eqref{e.easyste}, we deduce that for each $j\in\N$ with $R_j := 2^{-j} K \geq h$
\begin{equation*} \label{}
R_{j}^{-1} \osc_{B_{R_{j}}} \phi  \leq 7 \left( \left[ \phi \right]_{C^{1,1}_h(0,B_R)}\right)^{\frac12} \left(\osc_{B_R} \phi\right)^{\frac12}.
\end{equation*}
For general $s\in [h,R)$ we may find $j\in \N$ such that $R_{j+1} \leq s < R_j$ to get
\begin{equation*} \label{}
s^{-1} \osc_{B_s} \phi \leq R_{j+1}^{-1} \osc_{B_{R_j}} \phi \leq 2 R_{j}^{-1} \osc_{B_{R_{j}}} \phi \leq 14 \left(\left[ \phi \right]_{C^{1,1}_h(0,B_R)}\right)^{\frac12} \left(\osc_{B_R} \phi\right)^{\frac12}. \qedhere
\end{equation*}
\end{proof}

Equipped with this lemma, we now present the simple proof of Corollary \ref{c.C01reg}:
\begin{proof}[Proof of Corollary \ref{c.C01reg}]
By interpolation, we also obtain \eqref{e.pwC01}. This follows from~\eqref{e.pwC11} and~Lemma~\ref{l.interp} as follows:
\begin{align*}
\left[ u \right]_{C^{1,1}_h(0,B_R)} & \leq C \left[ u \right]^{\frac12}_{C^{1,1}_h(0,B_R)} \left(\osc_{B_R} u \right)^{\frac12} \\
& \leq C\X \left( K_0 + |f(0)| + R^{\sigma} \left[ f \right]_{C^{0,\sigma}(B_R)} + R^{-2} \osc_{B_R} u \right)^{\frac12}  \left(\osc_{B_R} u \right)^{\frac12}\\
& \leq C\X \left( K_0R +R |f(0)| + R^{1+\sigma} \left[ f \right]_{C^{0,\sigma}(B_R)} + R^{-1} \osc_{B_R} u \right),
\end{align*}
where we used~\eqref{e.interp} in the first line,~\eqref{e.pwC11} to get the second line and Young's inequality in the last line. Redefining $\X$ to absorb the constant $C$, we obtain~\eqref{e.pwC01}.
\end{proof}

\section{Green's function estimates}
\label{s.green}

We will now use a similar argument to the proof of Theorem \ref{t.regularity} to obtain estimates on the modified Green's functions $G_{\ep}(\cdot, 0)$ which are given by the solutions of:
\begin{equation}\label{e.greens2}
\ep^2 G_\ep -\tr\left(A (x) D^2G_\ep \right) = \chi_{B_\ell} \quad \mbox{in} \ \Rd.
\end{equation}

\begin{proposition}
\label{p.Green}
Fix $s\in (0,d)$. There exist $a(d,\lambda,\Lambda)>0$, $\delta(d,\lambda,\Lambda)>0$ and an $\F$--measurable random variable $\X:\Omega\to [1,\infty)$ satisfying
\begin{equation} \label{e.Green3dC}
\E \left[ \exp\left( \X^{s}  \right) \right] \leq C(s,d,\lambda,\Lambda,\ell) < \infty
\end{equation}
such that, for every $\ep \in (0,1]$ and $x\in \Rd$,
\begin{equation} \label{e.Greenest}
G_\ep(x,0) \leq \X^{d-1-\delta}\xi_{\ep}(x)
\end{equation}
where $\xi_{\ep}(x)$ is defined by:
\begin{equation}\label{e.dimsep}
\xi_{\ep}(x):= \exp\left( -a \ep|x| \right) \cdot \left\{ \begin{aligned} 
& \log\left( 2 + \frac{1}{\ep(1+|x|)} \right), & \mbox{in} \ d=2,\\
&  \left( 1 +  |x| \right)^{2-d}, & \mbox{in} \ d>2,
\end{aligned}\right.
\end{equation}

and
\begin{equation} \label{e.Greenoscest}
\osc_{B_1(x)} G_\ep(\cdot,0) \leq (T_x\X) \X^{d-1-\delta} \left( 1+|x| \right)^{1-d} \exp\left( -a \ep |x| \right).
\end{equation}
\end{proposition}

We emphasize that \eqref{e.Greenest} is a \emph{random} estimate, and the proof relies on the homogenization process. In contrast to the situation for divergence form equations, there is no \emph{deterministic} estimate for the decay of the Green's functions. Consider that for a general $A\in \Omega$, the Green's function $G(\cdot,0;A)$ solves
\begin{equation*} \label{}
-\tr\left( AD^2G(\cdot,0;A) \right) = \delta_0 \quad \mbox{in} \ \Rd.
\end{equation*}
The solution may behave, for $|x| \gg 1$, like a multiple of 
\begin{equation*} \label{}
K_\gamma(x) := \left\{ \begin{aligned} 
& |x|^{-\gamma} & \ \gamma > 0, \\
& \log |x| & \ \gamma = 0, \\
& - |x|^{-\gamma} & \ \gamma < 0, \\
\end{aligned}  \right.
\end{equation*}
for any exponent~$\gamma$ in the range
\begin{equation*} \label{}
 \frac{d-1}{\Lambda} -1 \leq \gamma \leq \Lambda(d-1) - 1.
\end{equation*}
In particular, if $\Lambda$ is large, then $\gamma$ may be negative and so $G(\cdot,0; A)$ may be bounded near the origin. To see that this range for $\gamma$ is sharp, it suffices to consider, respectively, the diffusion matrices
\begin{equation*} \label{}
A_1(x) = \Lambda \frac{x \otimes x}{|x|^2} + \left( I - \frac{x \otimes x}{|x|^2} \right) \quad \mbox{and} \quad A_2(x) =  \frac{x \otimes x}{|x|^2} + \Lambda \left( I - \frac{x \otimes x}{|x|^2} \right).
\end{equation*}
Note that $A_1$ and $A_2$ can be altered slightly to be smooth at $x=0$ without changing the decay of $G$ at infinity.

Before we discuss the proof of Proposition \ref{p.Green}, we mention an interesting application to the \emph{invariant measure} associated with \eqref{e.greens2}. Recall that the invariant measure is defined to be the solution $\im_\ep$ of the equation in double-divergence form:
\begin{equation*} \label{}
\ep^2(\im_\ep-1) - \div\left( D( A(x)\im_\ep ) \right) = 0 \quad \mbox{in} \ \Rd. 
\end{equation*}
By \eqref{e.Greenest}, we have that for every $y\in \mathbb{R}^{d}$, 
\begin{equation*} \label{}
\int_{B_{\ell}(y)} \im_\ep(x)\,dx \leq \int_{\Rd} G_{\ep}(x,0)\, dx\leq \X^{d-1-\delta}.
\end{equation*}
In particular, we deduce that, for some $\delta > 0$, 
\begin{equation*} \label{}
\P \left[ \int_{B_{\ell}(y)} \im_\ep(x)\,dx > t \right] \leq C\exp\left( -t^{\frac{d}{d-1}+\delta} \right).
\end{equation*}
This gives a very strong bound on the location of particles undergoing a diffusion in the random environment. 

\smallskip

We now return to the proof of Proposition \ref{p.Green}. Without loss of generality, we may change variables and assume that the effective operator $\overline{A}=I$. The proof of \eqref{e.Greenoscest} is based on the idea of using homogenization to compare the Green's function for the heterogeneous operator to that of the homogenized operator. The algebraic error estimates for homogenization in Proposition~\ref{p.subopt} are just enough information to show that, with overwhelming probability, the ratio of Green's functions must be bounded at infinity. This is demonstrated by comparing the modified Green's function $G_{\ep}(\cdot, 0)$ to a family of carefully constructed test functions. 

The test functions $\{ \varphi_R\}_{R\geq C}$ will possess the following three properties:
\begin{equation} \label{e.varphiRBR}
\inf_{A\in\Omega} -\tr\left( A(x) D^2 \varphi_R \right) \geq \chi_{B_\ell} \quad \mbox{in} \ B_R,
\end{equation}
\begin{equation} \label{e.varphiRBR2}
-\Delta \varphi_R (x) \gtrsim |x|^{-d} \quad \mbox{in} \ \Rd \setminus B_{R/2},
\end{equation}
\begin{equation} \label{e.varphiRBR3}
\varphi_R(x) \lesssim R^{d-1-\delta}(1+ |x|^{2-d}) \quad \mbox{in} \ \Rd\setminus B_R.
\end{equation}
As we will show, these properties imply, for large enough $R$ (which will be random and depend on the value of $\X$ from many different applications of Proposition~\ref{p.subopt}), that $G_\ep(\cdot,0) \leq \varphi_R$ in $\Rd$.

\smallskip

The properties of the barrier function $\vp_{R}$ inside and outside of $B_{R}$ will be used to compare with $G_{\ep}(\cdot, 0)$ in different ways.  If $G_\ep(\cdot,0) \not \leq \varphi_R$ then, since they both decay at infinity, $G_\ep(\cdot,0) - \varphi_R$ must achieve its global maximum somewhere in $\Rd$. Since $\varphi_R$ is a supersolution of~\eqref{e.varphiRBR}, this point must lie in $\Rd\setminus B_R$. As~$\varphi_R$ is a supersolution of the homogenized equation outside $B_{R/2}$, this event is very unlikely for $R\gg1$, by Proposition~\ref{p.subopt}. Note that there is a trade-off in our selection of the parameter $R$: if $R$ is relatively large, then $\varphi_R$ is larger and hence the conclusion $G_\ep(\cdot,0) \leq \varphi_R$ is weaker, however the probability that the conclusion fails is also much smaller.

\smallskip

Since the Green's function for the Laplacian has different qualitative behavior in dimensions $d=2$ and $d>2$, we split the proof of Proposition~\ref{p.Green} into these two cases, which are handled in the following subsections. 

\subsection{Proof of Proposition \ref{p.Green}: Dimensions three and larger}

\begin{lemma}\label{l.testfcnd3}
Let $s\in (0, d)$. Then there  exist constants $C, c, \gamma,\beta>0$, depending only on~$(s, d, \la, \La, \ell)$, and a family of continuous functions $\left\{\vp_{R}\right\}_{R\geq C}$ satisfying the following: (i) for every $R \geq C$ and $x\in\Rd$, 
\begin{equation} \label{e.varphimax}
\varphi_{R}(x) \leq C R^{d-2+\gamma} \left( 1 + |x|\right)^{2-d},
\end{equation}
(ii) there exists a smooth function $\psi_{R}$ such that 
\begin{equation}
\left\{ 
\begin{aligned}
& -\Delta \psi_{R}\geq  c|x|^{-2-\beta}\psi_{R} & \mbox{in} & \ \Rd\setminus B_{R/2}, \\
& \varphi_{R}\leq \psi_{R} & \mbox{in} & \ \Rd\setminus B_{R/2},
\end{aligned}
\right.
\end{equation}
and (iii) for each $R\geq C$ and $A\in \Omega$, we have 
\begin{equation} \label{e.eqvarphiR}
-\tr\left(A(x) D^2 \varphi_R \right) \geq \chi_{B_\ell} \quad \mbox{in} \ B_R.
\end{equation}
\end{lemma}

\begin{proof}
Throughout, we fix $s\in (0,d)$ and let $C$ and $c$ denote positive constants which depend only on $(s,d,\lambda,\Lambda,\ell)$ and may vary in each occurrence.


\smallskip

We define $\varphi_R$. For each $R \geq 4\ell$, we set
\begin{equation*} \label{}
\varphi_R(x):= \left\{ \begin{aligned}
& m_R - \frac{h}{\gamma} \left( \ell^2 + |x|^2 \right)^{\frac\gamma2}, && 0\leq |x| \leq R, \\
& k_R |x|^{2-d} \exp \left( -\frac1\beta |x|^{-\beta} \right),  && |x|> R, \\
\end{aligned} \right.
\end{equation*}
where we define the following constants:
\begin{equation*}
2\beta  :=  \alpha(s,d,\lambda,\Lambda,\ell) > 0 
\end{equation*}
is the exponent from Proposition~\ref{p.subopt} with $\sigma=1$,
\begin{align*}
 \gamma & := \max\left\{ \frac12 , 1-\frac{\lambda}{2\Lambda} \right\}  \\ 
 h &:= \frac{2}{\lambda} (2\ell)^{2-\gamma} \\
 k_R & := h \left(d-2-2^\beta R^{-\beta} \right)^{-1} R^{d-2+\gamma} \exp\left(\frac1\beta 2^\beta R^{-\beta} \right)  \\
 m_R & :=  \frac{h}{\gamma} \left( \ell^2 + R^2 \right)^{\frac\gamma2} + k_R R^{2-d} \exp \left( -\frac1\beta R^{-\beta} \right).
\end{align*}
Notice that the choice of $m_{R}$ makes $\vp_{R}$ continuous. We next perform some calculations to guarantee that this choice of $\vp_{R}$ satisfies the above claims. 
\smallskip

\emph{Step 1.}
We check that for every $R\geq 4\ell$ and $x\in \Rd$, ~\eqref{e.varphimax}~holds. Note that $\beta = \frac12\alpha \geq c$ and thus, for every $R \geq 4\ell$,
\begin{equation} \label{e.easyexp}
c\leq \exp\left( -\frac1\beta R^{-\beta}  \right) \leq 1.
\end{equation}
For such $R$, we also have that since $d\geq 3$, $(d-2-2^\beta R^{-\beta}) \geq c$. Morever, since $R \geq 4\ell \geq 4$, this implies that $(2/R)^\beta \leq 1-c$. Using also that $h\leq C$, we deduce that for every $R\geq 4\ell$,
\begin{equation} \label{e.easykas}
k_R \leq CR^{d-2+\gamma}\quad \mbox{and} \quad m_R \leq C R^\gamma .
\end{equation}
For $|x|>R$,~\eqref{e.varphimax} is immediate from the definition of $\varphi_R$,~\eqref{e.easyexp} and~\eqref{e.easykas}. For~$|x|\leq R$, we first note that $\varphi_R$ is a decreasing function in the radial direction and therefore $\sup_{\Rd} \varphi_R = \varphi_R(0) \leq m_R $. We then use~\eqref{e.easykas} to get, for every $|x| \leq R$,
\begin{equation*} \label{}
\varphi_R(x)  \leq m_R \leq CR^\gamma \leq C R^{d-2+\gamma} (1+|x|)^{2-d}.
\end{equation*}
This gives~\eqref{e.varphimax}.

\smallskip

\emph{Step 2.} We check that $\varphi_R$ satisfies
\begin{equation} \label{e.belowhalf}
\varphi_R(x) \leq \psi_R (x): = k_R |x|^{2-d} \exp\left( -\frac1\beta |x|^{-\beta} \right)\quad \mbox{in} \ \Rd\setminus B_{R/2}.
\end{equation}
Since this holds with equality for $|x| \geq R$, we just need to check it in the annulus $\{ R/2 \leq |x| < R\}$. For this it suffices to show that in this annulus, $\psi_R-\varphi_R$ is decreasing in the radial direction. Since both $\psi_R$ and $\varphi_R$ are decreasing radial functions, we simply need to check that
\begin{equation} \label{e.checkdec}
\left|D\varphi_R(x) \right| < \left|D\psi_R(x) \right| \quad \mbox{for every} \ x\in B_R \setminus B_{R/2}.
 \end{equation}
We compute, for $R/2\leq |x|\leq R$, since $\gamma\leq 1$, 
\begin{align*} \label{}
\left| D\varphi_R(x) \right| & = h \left( \ell^2 + |x|^2 \right)^{\frac\gamma2-1} |x| \leq h |x|^{\gamma-1}
\end{align*}
and
\begin{align*} \label{}
\left| D\psi_R(x) \right| & =|x|^{-1}\left(d-2-|x|^{-\beta}\right)\psi_{R}(x)\\
&= k_R \left(d-2-|x|^{-\beta}\right) |x|^{1-d} \exp\left( -\frac1\beta |x|^{-\beta} \right) \\
& \geq k_R \left(d-2-2^\beta R^{-\beta} \right) |x|^{1-d} \exp\left( -\frac1\beta 2^\beta R^{-\beta} \right).
\end{align*}
It is now evident that the choice of $k_R$ ensures that~\eqref{e.checkdec} holds. This completes the proof of~\eqref{e.belowhalf}.  

\smallskip

\emph{Step 3.} We check that $\psi_R$ satisfies 
\begin{equation} \label{e.psiR}
-\Delta \psi_R(x) \geq c |x|^{-2-\beta} \psi_R(x) \quad \mbox{in} \ |x| \geq C.
\end{equation}
By a direct computation, we find that, for $x\neq 0$,
\begin{align*} 
-\Delta \psi_R(x) & = |x|^{-2-\beta}  \left( d-2+\beta -|x|^{-\beta} \right)\psi_R(x).
\end{align*}
This yields~\eqref{e.psiR}. For future reference, we also note that for every $|x| >1$,
\begin{equation} \label{e.hesspsib}
|x|^{-2} \osc_{B_{|x|/2}(x)}  \psi_R + \sup_{y\in B_{|x|/2}(x)}\left( |y|^{-1} \left| D\psi_R(y) \right| \right) \leq C|x|^{-2} \psi_R(x).
\end{equation}
This follows from the computation
\begin{equation*} \label{}
\left| D\psi_R(x)\right| = |x|^{-1} \psi_R(x) \left( 2-d+|x|^{-\beta} \right).
\end{equation*}

\smallskip

\emph{Step 4.} We check that \eqref{e.eqvarphiR} holds. By a direct computation, we find that for every $x\in B_R$,
\begin{align*} \label{}
D^2 \varphi_R(x) &= -h\left(\ell^{2}+|x|^{2}\right)^{\frac{\gamma}{2}-1}\left(Id+\frac{\gamma-2}{\ell^{2}+|x|^{2}}\left(x\otimes x\right)\right)\\
&=-h \left( \ell^2+|x|^2  \right)^{\frac\gamma2-1} \left( \left( Id - \frac{x\otimes x}{|x|^2} \right) + \frac{\ell^2 - (1-\gamma)|x|^2}{\ell^2+|x|^2}\left( \frac{x\otimes x}{|x|^2} \right)   \right).
\end{align*}
Making use of our choice of $\gamma$, we see that, for any $A\in \Omega$ and $x\in B_R$,
\begin{align*} \label{}
-\tr\left( A(x) D^2 \varphi_R(x)  \right)
&  \geq h \left( \ell^2 +|x|^2 \right)^{\frac\gamma2 - 1} \left( (d-1)\lambda  -  \Lambda(1-\gamma)(\ell^2+|x|^2)^{-1}|x|^2  \right)   \\
& \geq h\left( \ell^2 +|x|^2 \right)^{\frac\gamma2 - 1} \left( (d-1)\lambda  -  \Lambda(1-\gamma)\right).
\end{align*}
The last expression on the right side is positive since, by the choice of $\gamma$, 
\begin{equation*} \label{}
(d-1)\lambda -  \Lambda(1-\gamma) \geq \left(d-\frac32\right)\lambda >\lambda > 0,
\end{equation*}
while for $x\in B_\ell$, we have, by the choice of $h$, 
\begin{align*} \label{}
h\left( \ell^2 +|x|^2 \right)^{\frac\gamma2 - 1} \left( (d-1)\lambda  -  \Lambda(1-\gamma)\right) \geq h\left( 2\ell^2 \right)^{\frac\gamma2 - 1} \lambda > 1.
\end{align*}
This completes the proof of~\eqref{e.eqvarphiR}. 
\end{proof}

\begin{proof}[Proof of Proposition~\ref{p.Green}~when $d\geq 3$]
As before, we fix $s\in (0,d)$ and let $C$ and $c$ denote positive constants which depend only on $(s,d, \la, \La, \ell)$. We use the notation developed in Lemma~\ref{l.testfcnd3} throughout the proof. 

\smallskip

We make one reduction before beginning the main argument. Rather than proving~\eqref{e.Greenest}, it suffices to prove
\begin{equation} \label{e.Greenest2}
\forall x\in\Rd, \quad 
G_\ep(x,0) \leq \X^{d-1-\delta} \left( 1 +  |x| \right)^{2-d}.
\end{equation}
To see this, we notice that 
\begin{equation*} \label{}
G_\ep(x,0) \leq \left( \sup_{|x|\leq\ep^{-1}} \frac{G_\ep(x,0)}{(1+|x|)^{2-d}} \right) \ep^{d-2} \exp(a)\exp\left( -a\ep|x| \right) \quad \mbox{in} \ \Rd \setminus B_{\ep^{-1}}.
\end{equation*}
Indeed, the right hand side is larger than the left hand side on $\partial B_{\ep^{-1}}$, and hence in $\Rd\setminus B_{\ep^{-1}}$ by the comparison principle and the fact that the right hand side is a supersolution of~\eqref{e.detsupersol} for $a(d,\lambda,\Lambda)>0$ (by the proof of Lemma~\ref{l.Gtails}). We then obtain~\eqref{e.Greenest} in $\Rd \setminus B_{\ep^{-1}}$ by replacing $\X$ by $C\X$ and $a$ by $\frac12a$, using~\eqref{e.Greenest2}, and noting that
\begin{equation*} \label{}
\ep^{d-2} \exp\left( -a\ep|x| \right) \leq C |x|^{2-d} \exp\left( -\frac a2 \ep |x| \right) \quad \mbox{for every} \ |x| \geq \ep^{-1}.
\end{equation*}
We also get~\eqref{e.Greenest} in $B_{\ep^{-1}}$, with $\X$ again replaced by $C\X$, from~\eqref{e.Greenest2} and the simple inequality
\begin{equation*} \label{}
\exp\left( -a\ep|x| \right) \geq c \quad \mbox{for every} \ |x| \leq \ep^{-1}.
\end{equation*}
\emph{Step 1.} We define $\X$ and check that it has the desired integrability. Let $\mathcal Y$ denote the random variable~$\X$ in the statement of Proposition~\ref{p.subopt} in $B_{R}$ with $s$ as above and $\sigma=1$. Also denote $\mathcal Y_x(A):= \mathcal Y(T_x A)$, which controls the error in balls of radius $R$ centered at a point $x\in\Rd$. 

We now define
\begin{equation} \label{e.defscrC}
\X(A) := \sup \left\{ |z|  \, :\, z\in \Zd, \ \mathcal Y_z(A) \geq 2^d|z|^s \right\}.
\end{equation}
The main point is that $\X$ has the following property by Proposition~\ref{p.subopt}: for every $z\in \Zd$ with $|z| > \X$, and every $R>\frac18|z|$ and $g\in C^{0,1} (\partial B_R(z))$, every pair $u,v\in C(\overline B_R)$ such that 
\begin{equation} \label{e.X1defp1}
\left\{  \begin{aligned} 
& - \tr(A(x)D^2u) \leq 0 \leq -\Delta v & \mbox{in} & \ B_R(z), \\
& u \leq g \leq v & \mbox{on} & \ \partial B_R(z),
\end{aligned}
\right.
\end{equation}
must satisfy the estimate
\begin{equation} \label{e.X1defp2}
R^{-2}\sup_{B_R(z)} \left( u(x) - v(x) \right) \leq CR^{-\alpha} \left( R^{-2} \osc_{\partial B_{R}(z)} g + R^{-1} \left[ g \right]_{C^{0,1}(\partial B_{R}(z))} \right).
\end{equation}

Let us check that 
\begin{equation} \label{e.expC1m}
\E \left[ \exp\left(\X^s\right)\right] \leq C(s,d,\lambda, \Lambda,\ell) < \infty. 
\end{equation}
A union bound and stationarity yield, for $t\geq 1$, 
\begin{align*} \label{}
\P \left[ \X > t \right] & \leq \sum_{z\in \Zd \setminus B_t} \P \left[ \mathcal Y_z \geq 2^d|z|^s
 \right]
 \\ & \leq \sum_{n\in \N, \, 2^{n} \geq t} \ \sum_{z\in B_{2^n} \setminus B_{2^{n-1}}} \P \left[ \mathcal Y_z \geq 2^d|z|^s
 \right] \\ 
 & \leq C \sum_{n\in \N, \, 2^{n} \geq t} 2^{dn} \,\P \left[\mathcal Y \geq 2^{(n-1)s+d} \right]. 
\end{align*}
By Proposition~\ref{p.subopt} and Chebyshev's inequality,
\begin{align*} \label{}
\sum_{n\in \N, \, 2^{n}  \geq t} 2^{dn} \,\P \left[\mathcal Y \geq 2^{(n-1)s+d} \right] & \leq C\sum_{n\in \N, \, 2^{n} \geq t} 2^{dn}  \exp\left(-2^{(n-1)s+d} \right) \\
& \leq C \exp\left( -2t^s \right).
\end{align*}
It follows then that 
\begin{align*}
\E[\exp(\X^{s})]&=s\int_{0}^{\infty} t^{s-1}\exp(t^{s})\P[\X>t]\, dt\\
&\leq sC\int_{0}^{\infty}t^{s-1} \exp(-t^{s}) \,dt\leq C.
\end{align*}
This yields~\eqref{e.expC1m}.

\smallskip

\emph{Step 2.}
We reduce the proposition to the claim that, for every $R\geq C$,
\begin{equation}
\label{e.domclam2}
\left\{ A\in \Omega \,:\, \sup_{0<\ep\leq 1} \sup_{x\in \Rd} \left( G_\ep(x,0; A) - \varphi_R(x) \right)  >0 \right\} \subseteq \left\{ A \in\Omega \,:\, \X(A) > R \right\}.
\end{equation}
If~\eqref{e.domclam2} holds, then by~\eqref{e.varphimax} we have
\begin{multline*}
\left\{ A\in \Omega \,:\,  \sup_{0<\ep\leq1}\sup_{x\in\Rd} \left( G_\ep(x,0; A) - C R^{d-2+\gamma} \left( 1+ |x|\right)^{2-d}\right) >0  \right\} \\
 \subseteq \left\{ A\in\Omega\,:\, \X(A) > R \right\}.
\end{multline*}
However this implies that, for every $R\geq C$, $0<\ep \leq 1$ and $x\in\Rd$,
\begin{equation*} \label{}
G_\ep(x,0) \leq C \X^{d-2+\gamma} \left(1+ |x|\right)^{2-d}.
\end{equation*}
Setting~$\delta := 1-\gamma \geq c(d,\lambda,\Lambda)>0$, we obtain~\eqref{e.Greenest2}.

\smallskip

\emph{Step 3.} We prove~\eqref{e.domclam2}. Fix $A\in \Omega$, $0< \ep\leq 1$ and $R\geq 10\sqrt{d}$ for which 
\begin{equation*} \label{}
\sup_{\Rd} \left( G_\ep(\cdot,0) - \varphi_R \right) > 0. 
\end{equation*}
The goal is to show that $\X \geq R$, at least if $R\geq C$. We do this by exhibiting $|z| > R$ and functions $u$ and $v$ satisfying~\eqref{e.X1defp1}, but not~\eqref{e.X1defp2}. 

\smallskip

As the functions $G_\ep(\cdot,0)$ and $\varphi_R$ decay at infinity (c.f. Lemma~\ref{l.Gtails}), there exists a point $x_0 \in \Rd$ such that 
\begin{equation*} \label{}
G_\ep(x_0,0) - \varphi_R(x_0) = \sup_{\Rd} \left( G_\ep(\cdot,0) - \varphi_R \right) > 0.
\end{equation*}
By the maximum principle and~\eqref{e.eqvarphiR}, it must be that $|x_0| \geq R$. By~\eqref{e.belowhalf},
\begin{equation} \label{e.Greent1}
G_\ep(x_0,0) - \psi_R(x_0) = \sup_{B_{|x_0|/2}(x_0)} \left( G_\ep(\cdot,0) - \psi_R \right).
\end{equation}
We perturb $\psi_R$ by setting $\widetilde \psi_R (x):= \psi_R(x) + c|x_0|^{-2-\beta} \psi_R(x_0)|x-x_{0}|^{2}$ which, in view of~\eqref{e.psiR}, satisfies
\begin{equation*} \label{}
-\Delta \widetilde \psi_R \geq 0 \quad \mbox{in} \  B_{|x_0|/2}(x_0).
\end{equation*}
The perturbation improves~\eqref{e.Greent1} to
\begin{equation*} 
G_\ep(x_0,0) - \tilde \psi_R(x_0) \geq \sup_{\partial B_{|x_0|/2}(x_0)} \left( G_\ep(\cdot,0) - \tilde \psi_R \right) + c|x_0|^{-\beta} \psi_R(x_0).
\end{equation*}
If $R\geq C$, then we may take $z_0 \in \Zd$ to be the nearest lattice point to $x_0$ such that $|z_0| > |x_0|$ and get $x_0\in B_{|z_0| /4}(z_{0})$. Since $\psi(|x|)$ is decreasing in $|x|$, this implies
\begin{equation*}
G_\ep(x_0,0) - \tilde \psi_R(x_0) \geq \sup_{\partial B_{|z_0|/4}(z_0)} \left( G_\ep(\cdot,0) - \tilde \psi_R \right) + c|z_0|^{-\beta} \psi_R(z_0).
\end{equation*}
In view of~\eqref{e.hesspsib}, this gives
\begin{equation*} \label{}
G_\ep(x_0,0) - \tilde \psi_R(x_0) \geq \sup_{\partial B_{|z_0|/4}(z_0)} \left( G_\ep(\cdot,0) - \tilde \psi_R \right) + c \Gamma |z_0|^{-\beta},
\end{equation*}
where 
\begin{equation*} \label{}
\Gamma: = \osc_{\partial B_{|z_0|/4}(z_0)} \tilde \psi_R + |z_0| \big[ \tilde \psi_R \big]_{C^{0,1}(\partial B_{|z_0|/4}(z_0))}.
\end{equation*}
We have thus found functions satisfying~\eqref{e.X1defp1} but in violation of~\eqref{e.X1defp2}. That is, we deduce from the definition of $\mathcal Y_z$ that  $\mathcal Y_{z_0} \geq c|z_0|^{s+\alpha-\beta}-C$ and, in view of the fact that $\beta =\frac12\alpha< \alpha$ and $|z_0|>R$, this implies that $\mathcal Y_{z_0} \geq 2^d|z_0|^{s}$ provided $R \geq C$. Hence $\X \geq |z_0| > R$. This completes the proof of~\eqref{e.domclam2}. 
\end{proof}

\subsection{Dimension two}

The argument for~\eqref{e.Greenest} in two dimensions follows along similar lines as the proof when $d\geq 3$, however the normalization of $G_\ep$ is more tricky since it depends on $\ep$. 

\begin{lemma}\label{l.testfcnd2}
Let $s\in (0, 2)$. Then there exist constants $C, c, \gamma, \beta>0$, depending only on $(s, d, \la, \La, \ell)$, and a family of continuous functions $\left\{\vp_{R, \ep}\right\}_{R\geq C, \ep\leq c}$ satisfying the following: (i) for every $R\geq C$,
\begin{equation} \label{e.varphiepbnds2}
\varphi_{R,\ep}(x) \leq CR^\gamma \log\left( 2 + \frac{1}{\ep(1+|x|)} \right) \exp\left( - a \ep  |x| \right),
\end{equation}
(ii) there exists a smooth function $\psi_{R, \ep}$ such that 
\begin{equation*}
\left\{
\begin{aligned}
& -\Delta \psi_{R, \ep}\geq  c|x|^{-2-\beta}\psi_{R,\ep} & \mbox{in} & \ B_{2\ep^{-1}}\setminus B_{R/2}, \\
& \varphi_{R, \ep}\leq \psi_{R, \ep} & \mbox{in} & \ B_{2\ep^{-1}}\setminus B_{R/2},
\end{aligned}
\right.
\end{equation*}
 and (iii) for every $R\ge C$ and $A\in\Omega$,
\begin{equation} \label{e.eqvarphiR2}
-\tr\left(A(x) D^2 \varphi_{R, \ep} \right) \geq \chi_{B_\ell} \quad \mbox{in} \ B_R\bigcup\left(\R^{2}\setminus B_{\ep^{-1}}\right) .
\end{equation}
\end{lemma}

\begin{proof}
Throughout we assume $d=2$, we fix $s\in (0,2)$ and let $C$ and $c$ denote positive constants which depend only on $(s,\lambda,\Lambda,\ell)$ and may vary in each occurrence. We roughly follow the outline of the proof of~\eqref{e.Greenest} above in the case of $d\geq 3$. 

\smallskip

\emph{Step 1.} The definition of $\varphi_R$. For $\ep\in (0,\frac12]$, $4\ell\leq R\leq \ep^{-1}$  and $x\in\R^2$, we set
\begin{equation*} \label{}
\varphi_{R,\ep}(x):= \left\{ \begin{aligned}
& m_{R,\ep} - \frac{h}{\gamma} \left( \ell^2 + |x|^2 \right)^{\frac\gamma2}, && 0\leq |x| \leq R, \\
& k_{R}\left( \frac1a \exp(a) + \left| \log \ep \right| - \log |x| \right)  \exp \left( -\frac1\beta |x|^{-\beta} \right),  && R< |x| \leq \frac 1\ep, \\
& b_{R,\ep} \exp\left( -a\ep |x|\right), &&  |x| > \frac 1\ep,
\end{aligned} \right.
\end{equation*}
where the constants are defined as follows:
\begin{align*}
 a & : = a(\lambda,\Lambda)  \ \ \mbox{is the constant from Lemma~\ref{l.Gtails},}\\
2\beta & := \lefteqn{ \alpha(s,\lambda,\Lambda,\ell) > 0 \ \ \mbox{is the exponent from Proposition~\ref{p.subopt} with $\sigma=1$,}} \\
 \gamma & := \max\left\{ \frac12 , 1-\frac{\lambda}{2\Lambda} \right\},  \\ 
 h &:= \frac{2}{\lambda} (2\ell)^{2-\gamma},\\
 k_R & := 2h \exp\left(\frac1\beta 2^\beta R^{-\beta} \right) R^{\gamma} , \\
 m_{R,\ep} & :=  \frac{h}{\gamma} \left( \ell^2 + R^2 \right)^{\frac\gamma2} + k_R \left( \frac1a\exp(a) + \left| \log \ep \right| - \log R \right) \exp\left(-\frac1\beta R^{-\beta} \right),\\
 b_{R,\ep} & :=\frac{1}{a}k_R \exp\left(2a -\frac1\beta \ep^{\beta} \right)
\end{align*}
Observe that
\begin{equation} \label{e.easy2bnds}
k_R \leq CR^\gamma, \quad m_{R,\ep} \leq C R^\gamma \left( 1+ \left| \log \ep \right| - \log R \right) \quad \mbox{and} \quad b_{R,\ep} \leq CR^\gamma.
\end{equation}

\smallskip

\emph{Step 2.} We show that, for every $\ep \in(0, \frac12]$, $4\ell \leq R\leq \ep^{-1}$ and $x\in\R^2$,~\eqref{e.varphiepbnds2} holds. This is relatively easy to check from the definition of $\varphi_{R,\ep}$, using~\eqref{e.easyexp} and~\eqref{e.easy2bnds}. For $x\in B_R$, we use $\varphi_{R,\ep} \leq m_{R,\ep}$,~\eqref{e.easy2bnds} and $\exp(-a\ep R) \geq c$ to immediately obtain~\eqref{e.varphiepbnds2}. For $x\in B_{\ep^{-1}}\setminus B_R$, the estimate is obtained from the definition of $\varphi_{R,\ep}$, the bound for $k_R$ in~\eqref{e.easy2bnds} and~\eqref{e.easyexp}. For $x\in \mathbb{R}^{2}\setminus B_{\ep^{-1}}$, the logarithm factor on the right side of~\eqref{e.varphiepbnds2} is $C$ and we get~\eqref{e.varphiepbnds2} from the bound for $b_{R,\ep}$ in~\eqref{e.easy2bnds}.

\smallskip

\emph{Step 3.} We show that, for $\ep \in (0,c]$ and $R\geq C$, we have
\begin{multline} \label{e.midrngs}
\varphi_{R,\ep} (x) \leq \psi_{R,\ep} (x) := k_{R}\left( \frac1a\exp(a) + \left| \log \ep \right| - \log |x| \right)  \exp \left( -\frac1\beta |x|^{-\beta} \right) \\ \mbox{in} \ \frac12 R \leq |x| \leq \frac{2}{\ep}.
\end{multline}
We have equality in~\eqref{e.midrngs} for $R \leq |x| \leq \ep^{-1}$ by the definition of~$\varphi_{R,\ep}$. As $\varphi_{R,\ep}$ is radial, it therefore suffices to check that the magnitude of the radial derivative of $\varphi_{R,\ep}$ is less than (respectively, greater than) than that of $\psi_{R,\ep}$ in the annulus $\{ R/2\leq |x| \leq R\}$ (respectively, $\{ \ep^{-1} \leq |x| \leq 2\ep^{-1}\}$). This is ensured by the definitions of $k_R$ and $b_{R,\ep}$, as the following routine computation verifies: first, in $x\in B_R \setminus B_{R/2}$, we have 
\begin{align*}
\left| D\varphi_{R,\ep}(x) \right| =  h\left( \ell^2 +|x|^2 \right)^{\frac\gamma2-1} |x| < h |x|^{\gamma-1},
\end{align*}
and thus in $B_{R}\setminus B_{R/2}$, provided $R\geq C$, we have that 
\begin{align*}
\left| D\psi_{R,\ep}(x) \right| & = k_R |x|^{-1} \left| -1+|x|^{-\beta} \left( \frac1a\exp(a) + \left| \log \ep \right| - \log |x|  \right) \right|\exp\left( -\frac1\beta |x|^{-\beta} \right) \\
& > \frac{1}{2}k_R |x|^{-1} \exp\left( -\frac1\beta 2^\beta R^{-\beta}  \right) = h R^{\gamma} |x|^{-1} \geq h |x|^{\gamma-1} > \left| D\varphi_{R,\ep}(x) \right|.
\end{align*}
Next we consider $x\in B_{2\ep^{-1}} \setminus B_{\ep^{-1}}$ and estimate
\begin{align*}
\left| D\varphi_{R,\ep}(x) \right| = a\ep b_{R,\ep} \exp\left( -a\ep|x| \right) > a \ep b_{R,\ep} \exp\left( -2a \right) = 2\ep k_R \exp\left(-\frac{1}{\beta}\ep^{\beta}\right)
\end{align*}
and
\begin{align*}
\left| D\psi_{R,\ep}(x) \right| & \leq k_R |x|^{-1} \left(1+ \frac{1}{a}\exp(a)|x|^{-\beta}\right)\exp\left(-\frac{1}{\beta}\frac{\ep^{\beta}}{2^{\beta}}\right)\leq 2\ep k_{R}\exp\left(-\frac{1}{\beta}\ep^{\beta}\right), 
\end{align*}
the latter holding provided that $\ep \leq c$. This completes the proof of~\eqref{e.midrngs}.

\smallskip

\emph{Step 4.} We show that $\psi_{R,\ep}$ satisfies
\begin{equation} \label{e.psiRep}
-\Delta \psi_{R,\ep} \geq c |x|^{-2-\beta} \psi_{R,\ep}(x) \quad \mbox{in} \ C \leq |x| \leq \frac{2}{\ep}.
\end{equation}
By a direct computation, for every $x\in \R^2\setminus\{ 0\}$, we have
\begin{align*} \label{}
-\Delta \psi_{R,\ep}(x) & = |x|^{-2-\beta} \left(  \left(\beta - |x|^{-\beta}  \right) \psi_{R,\ep}(x) + k_R\left(\frac{1}{|x|^{2}}+1\right) \exp\left( -\frac1\beta|x|^{-\beta} \right) \right) \\
& \geq |x|^{-2-\beta}\left(\beta - |x|^{-\beta}  \right) \psi_{R,\ep}(x).
\end{align*}
From $\beta \geq c$ and the definition of~$\psi_{R,\ep}$, we see that $\psi_{R,\ep} >0$ and $(\beta-|x|^{-\beta}) \geq c$ for every $\beta^{-1/\beta} \leq |x| \leq 2\ep^{-1}$. This yields~\eqref{e.psiRep}.

For future reference, we note that, for every $|x| \leq 2\ep^{-1}$,
\begin{equation} \label{e.blaggber}
|x|^{-2} \osc_{B_{|x|/2}(x)} \psi_{R,\ep} + \sup_{y\in B_{|x|/2}(x)} \left( |y|^{-1} D\psi_{R,\ep}(y) \right) \leq C |x|^{-2} \leq C |x|^{-2} \psi_{R,\ep}(x).
\end{equation}

\smallskip

\emph{Step 5.} We check ~\eqref{e.eqvarphiR2} by checking that for every $A\in \Omega$, 
\begin{equation}\label{e.eqvarphiR2*}
-\tr\left(A(x) D^2 \varphi_{R,\ep} \right) \geq \chi_{B_{\ell}} \quad \mbox{in} \ B_R \cup (\R^2\setminus B_{\ep^{-1}}).
\end{equation}
In fact, for $B_R$, the computation is identical to the one which established~\eqref{e.eqvarphiR}, since our function $\varphi_{R,\ep}$ here is the same as $\varphi_R$ (from the argument in the case $d>2$) in $B_R$, up to a constant. Therefore we refer to Step 5 in Lemma~\ref{l.testfcnd3}~for details. In $\R^{2}\setminus B_{\ep^{-1}}$, $\varphi_{R, \ep}$ is a supersolution by the proof of Lemma \ref{l.Gtails} and by the choice of $a$. 

 We remark that in the case that $R= \ep^{-1}$, the middle annulus in the definition of $\varphi_{R,\ep}$ disappears, and we have that $\varphi_{R,\ep}$ is a global (viscosity) solution of~\eqref{e.eqvarphiR2*}, that is, 
\begin{equation} \label{e.eqvarphiR2spec}
-\tr\left(A(x) D^2 \varphi_{R,\ep} \right) \geq \chi_{B_{\ell}} \quad \mbox{in} \ \R^2 \quad \mbox{if} \ R = \ep^{-1}.
\end{equation}

To see why, by~\eqref{e.eqvarphiR2}, we need only check that $\varphi_{R,\ep}$ is a viscosity supersolution of~\eqref{e.eqvarphiR2} on the sphere $\partial B_R  = \partial B_{\ep^{-1}}$. However, the function $\varphi_{R,\ep}$ cannot be touched from below on this sphere, since its inner radial derivative is smaller than its outer radial derivative by the computations in Step 3. Therefore we have~\eqref{e.eqvarphiR2spec}. It follows by comparison that, for every $\ep \in (0,\frac12]$ and $x\in \R^2$,
\begin{align} \label{e.taut}
G_\ep(x,0) &\leq \varphi_{R,\ep}(x) \quad \mbox{if} \ R = \ep^{-1}\notag\\
&\leq C\ep^{-\gamma} \log\left( 2 + \frac{1}{\ep(1+|x|)} \right) \exp\left( - a \ep  |x| \right)
\end{align}

\end{proof}

\begin{proof}[Proof of Proposition~\ref{p.Green}~when $d=2$]
The proof follows very similar to the case when $d\geq 3$, using the appropriate adaptations for the new test function introduced in Lemma~\ref{l.testfcnd2}. As before, we fix $s\in (0,2)$ and let $C$ and $c$ denote positive constants which depend only on $(s, \la, \La, \ell)$. We use the notation developed in Lemma~\ref{l.testfcnd2} throughout the proof. 

\smallskip

\emph{Step 1.} We define the random variable $\X$ in exactly the same way as in \eqref{e.defscrC}, so 
\begin{equation}\label{e.defscrC2}
\X(A) := \sup \left\{ |z|  \, :\, z\in \Zd, \ \mathcal Y_z(A) \geq 2^d|z|^s \right\},
\end{equation}
The argument leading to \eqref{e.expC1m} follows exactly as before, so that 
\begin{equation*}
\E[ \exp(\X^{s}) ]\leq C(s, \la, \La, \ell)<\infty.
\end{equation*}

\smallskip

\emph{Step 2.} We reduce the proposition to the claim that, for every $R\geq C$,
\begin{multline}
\label{e.domclam22d}
\left\{ A\in \Omega \,:\, \sup_{0<\ep< \frac1R} \sup_{x\in \R^2} \left( G_\ep(x,0,A) - \varphi_{R,\ep}(x) \right)  >0 \right\} \\ \subseteq \left\{ A \in \Omega \,:\, \X (A) > R \right\}.
\end{multline}
If~\eqref{e.domclam22d} holds, then by~\eqref{e.varphiepbnds2} we have
\begin{align*}
\left\{ A \,:\,  \sup_{0<\ep<\frac1R}\ \sup_{x\in\R^2} \left( G_\ep(x,0,A) -  CR^\gamma \log\left( 2 + \frac{1}{\ep(1+|x|)} \right) \exp\left( - a \ep  |x| \right)\right) >0  \right\} \\
 \subseteq  \left\{ A\in\Omega\,:\, \X(A) > R \right\}.
\end{align*}
From this, we deduce that, for every $0<\ep < \X^{-1}$ and $x\in\R^2$,
\begin{equation} \label{e.Gep2bnd}
G_\ep(x,0) \leq C \X^{\gamma}  \log\left( 2 + \frac{1}{\ep(1+|x|)} \right) \exp\left( - a \ep  |x| \right).
\end{equation}
Moreover, if $\ep \in (0,\frac12]$ and $\ep \geq \X^{-1}$, then by ~\eqref{e.taut} we have
\begin{align*}
G_\ep(x,0) & \leq C \ep^{-\gamma}  \log\left( 2 + \frac{1}{\ep(1+|x|)} \right) \exp\left( - a \ep  |x| \right) \\
& \leq C \X^{\gamma}  \log\left( 2 + \frac{1}{\ep(1+|x|)} \right) \exp\left( - a \ep  |x| \right).
\end{align*}
Thus we have~\eqref{e.Gep2bnd} for every $\ep \in (0,\frac12]$ and $x\in\R^2$. Taking $\delta := 1-\gamma \geq c$, we obtain the desired conclusion~\eqref{e.Greenest} for $d=2$.

\smallskip

\emph{Step 3.} We prove~\eqref{e.domclam22d}. This is almost the same as the last step in the proof of~\eqref{e.domclam2} for dimensions larger than two. Fix $A\in \Omega$, $0< \ep\leq 1$ and $R\geq 2\ell$ for which 
\begin{equation*} \label{}
\sup_{\Rd} \left( G_\ep(\cdot,0) - \varphi_{R,\ep} \right) > 0. 
\end{equation*}
As in the case of dimensions larger than two, we select a point $x_0 \in \Rd$ such that 
\begin{equation*} \label{}
G_\ep(x_0,0) - \varphi_{R,\ep}(x_0) = \sup_{\Rd} \left( G_\ep(\cdot,0) - \varphi_{R,\ep} \right) > 0.
\end{equation*}
By the maximum principle and~\eqref{e.eqvarphiR2}, it must be that $R \leq |x_0| \leq \ep^{-1}$. By~\eqref{e.midrngs},
\begin{equation} \label{e.Greent12d}
G_\ep(x_0,0) - \psi_{R,\ep}(x_0) = \sup_{B_{|x_0|/2}(x_0)} \left( G_\ep(\cdot,0) - \psi_{R,\ep} \right).
\end{equation}
We perturb $\psi_{R,\ep}$ by setting $\widetilde \psi_{R,\ep} (x):= \psi_{R,\ep}(x) + c|x_0|^{-2-\beta} \psi_{R,\ep}(x_0) |x-x_0|^2$ which, in view of~\eqref{e.psiRep}, satisfies
\begin{equation*} \label{}
-\Delta \widetilde \psi_{R,\ep} \geq 0 \quad \mbox{in} \  B_{|x_0|/2}(x_0).
\end{equation*}
According to~\eqref{e.Greent1}, we have
\begin{equation*}
G_\ep(x_0,0) - \tilde \psi_{R,\ep}(x_0) \geq \sup_{\partial B_{|x_0|/2}(z_0)} \left( G_\ep(\cdot,0) - \tilde \psi_{R,\ep} \right) + c|x_0|^{-\beta} \psi_{R,\ep}(x_0).
\end{equation*}
Assuming $R\geq C$, we may take $z_0 \in \Zd$ to be the nearest lattice point to $x_0$ such that $|z_0| > |x_0|$ and deduce that $x_0\in B_{|z_0| /4}$ as well as
\begin{equation*}
G_\ep(x_0,0) - \tilde \psi_{R,\ep}(x_0) \geq \sup_{\partial B_{|z_0|/4}(z_0)} \left( G_\ep(\cdot,0) - \tilde \psi_{R,\ep} \right) + c|z_0|^{-\beta} \psi_{R,\ep}(z_0).
\end{equation*}
In view of~\eqref{e.blaggber}, this gives
\begin{equation*} \label{}
G_\ep(x_0,0) - \tilde \psi_{R,\ep}(x_0) \geq \sup_{\partial B_{|z_0|/4}(z_0)} \left( G_\ep(\cdot,0) - \tilde \psi_{R,\ep} \right) + c \Gamma |z_0|^{-\beta},
\end{equation*}
where 
\begin{equation*} \label{}
\Gamma: = |z_0|^{-2} \osc_{\partial B_{|z_0|/4}(z_0)} \tilde \psi_{R,\ep} + |z_0|^{-1} \big[ \tilde \psi_{R,\ep} \big]_{C^{0,1}(\partial B_{|z_0|/4}(z_0))}.
\end{equation*}
We have thus found functions satisfying~\eqref{e.X1defp1} but in violation of~\eqref{e.X1defp2}, that is, we deduce from the definition of $\mathcal Y_z$ that  $\mathcal Y_{z_0} \geq c|z_0|^{s+\alpha-\beta}-C$. In view of the fact that $\beta =\frac12\alpha< \alpha$ and $|z_0|>R$, this implies that $\mathcal Y_{z_0} \geq 2^d|z_0|^{s}$ provided $R \geq C$. By the definition of $\X$, we obtain $\X \geq |z_0| > R$. This completes the proof of~\eqref{e.domclam22d}. 
\end{proof}

\section{Sensitivity estimates}
\label{s.sensitivity}
In this section, we present an estimate which uses the Green's function bounds to control the vertical derivatives introduced in the spectral gap inequality (Proposition \ref{p.concentration}). Recall the notation from Proposition~\ref{p.concentration}:
\begin{equation*} \label{}
X_z' := \E_* \left[ X \,\vert\, \F_*(\Zd\setminus \{ z\}) \right] \qquad \mbox{and} \qquad \V_*\!\left[ X \right] := \sum_{z\in\mathbb{Z}^{d}} (X-X'_z)^2.
\end{equation*}
The vertical derivative $(X-X'_{z})$ measures, in a precise sense, the \emph{sensitivity} of $X$ subject to changes in the environment near $z$. We can therefore interpret $\left(X-X_{z}'\right)$ as a derivative of $X$ with respect to the coefficients near $z$. The goal then will be to understand the vertical derivative $(X-X_z')$ when $X$ is $\phi_{\ep}(x)$, for fixed $x\in \mathbb{R}^{d}$. 
\smallskip

The main result of this section is the following proposition which computes $\phi_\ep(x) - \E_*\!\left[ \phi_\ep(x) \,\vert\, \F_*(\Zd\setminus \{ z \})\right]$ in terms of the random variable introduced in Proposition \ref{p.Green}. Throughout the rest of the section, we fix $M\in \mathbb{S}^{d}$ with $|M|=1$ and let $\xi_{\ep}$ be defined as in \eqref{e.dimsep}. 

\begin{proposition}\label{p.sensitivity}
Fix $s\in (0,d)$. There exist positive constants $a(d,\lambda,\Lambda)>0$, $\delta(d,\lambda,\Lambda)>0$  and an $\F_*$--measurable random variable $\X:\Omega_* \to [1,\infty)$ satisfying
\begin{equation} \label{e.scruff}
\E \left[ \exp(\X^s) \right] \leq C(s,d,\lambda,\Lambda,\ell) < \infty
\end{equation}
such that, for every $\ep\in \left(0,\tfrac12\right]$, $x\in\Rd$ and $z\in \Zd$, 
\begin{equation}\label{e.Vbound1}
\left| \phi_\ep(x)\right.\left. - \E_*\!\left[ \phi_\ep(x) \,\vert\, \F_*(\Zd\setminus \{ z \}) \right] \right|
\leq (T_z\X)^{d+1-\delta}  \xi_{\ep}(x-z).
\end{equation} 
\end{proposition}

Before beginning the proof of Proposition~\ref{p.sensitivity}, we first provide a heuristic explanation of the main argument. We begin with the observation that we may identify the conditional expectation $\E_*\!\left[ X \,\vert\, \F_*(\Zd \setminus \{ z \}) \right]$ via resampling in the following way. Let $(\Omega_*',\F_*', \P_*')$ denote an independent copy of $(\Omega_*,\F_*, \P_*)$ and define, for each $z\in\Zd$, a map
\begin{equation*} \label{}
\theta_z':\Omega\times\Omega'\to \Omega
\end{equation*}
by
\begin{equation*} \label{}
\theta_z'(\omega,\omega')(y):= \left\{ \begin{aligned} & \omega(y) & \mbox{if} \ y\neq z,\\
& \omega'(z) &  \mbox{if} \ y=z.
\end{aligned}\right.
\end{equation*}
It follows that, for every $\omega\in\Omega$,
\begin{equation} \label{e.resampling}
\E_*\left[ X \,\vert\, \F_*(\Zd\setminus \{ z \}) \right](\omega) = \E_*'\left[  X(\theta_z'(\omega,\cdot))\right].
\end{equation}

Therefore, we are interested in estimating differences of the form $X(\omega) - X( \theta'_z(\omega,\omega'))$, which represent the expected change in $X$ if we resample the environment at~$z$. Observe that, by ~\eqref{e.siginclusion}, if $\omega,\omega' \in\Omega_*$, $z\in \Rd$, and $A := \pi(\omega)$ and $A':=\pi(\theta_z'(\omega,\omega'))$, then
\begin{equation} \label{e.AtotA}
A \equiv A'\quad \mbox{in} \ \Rd\setminus B_{\ell/2}(z).
\end{equation}
Denote by~$\phi_\ep$ and~$\phi_\ep'$ the corresponding approximate correctors with modified Green's functions~$G_\ep$ and~$G_\ep'$. Let $w:= \phi_\ep -  \phi_\ep'$. Then we have
\begin{align*} \label{}
\ep^2 w - \tr\left(  A'(x) D^2w \right) = \tr\left( \left( A(x) -  A'(x) \right)\right.&\left.(M+D^2\phi_\ep) \right)  \\
&\leq d\Lambda\left(1+ \left| D^2 \phi_\ep(x) \right|\right) \chi_{B_{\ell/2}(z)}(x),
\end{align*}
where~$\chi_E$ denotes the characteristic function of a set~$E\subseteq\Rd$. By comparing~$w$ to~$G_\ep'$, we deduce that, for $C(d,\lambda,\Lambda)\geq1$,
\begin{equation} \label{e.rescomp}
\phi_\ep(x) -  \phi_\ep'(x) \leq C(1+\left[ \phi_\ep \right]_{C^{1,1}(B_{\ell/2}(z))}) G_\ep'(x,z).
\end{equation}

If $\phi_{\ep}$ satisfied a $C^{1,1}$ bound, then by~\eqref{e.pwcC11} and ~\eqref{e.Greenest}, we deduce that
\begin{equation*} \label{}
\phi_\ep(0) -  \phi_\ep'(0) \leq C \left((T_z\X)(\omega)\right)^2 \left((T_z\Y)(\theta'_z(\omega,\omega'))\right)^{d-1-\delta} \xi_{\ep}(z),
\end{equation*}
for $\X$ defined as in \eqref{e.pwcC11}, $\Y$ defined as in \eqref{e.Greenest}, and $\xi_{\ep}(z)$ defined as in~\eqref{e.dimsep}.

Taking expectations of both sides with respect to $\P_*'$, we obtain
\begin{equation} \label{e.wtssense}
 \phi_\ep(0) - \E_*\!\left[ \phi_\ep(0) \,\vert\, \F_*(\Zd\setminus \{ z \}) \right] \leq C (T_z\X)^2 (T_z\Y_*)^{d-1-\delta}\xi_{\ep}(z),
\end{equation}
where 
\begin{equation} \label{e.Ysrv}
\Y_*:= \E_*' \left[ \Y(\theta'_0(\omega,\omega'))^{d-1-\delta} \right]^{1/(d-1-\delta)}.
\end{equation}

Jensen's inequality implies that the integrability of $\Y_*$ is controlled by the integrability of $\Y$. First, consider that, for $s\geq d-1-\delta$, by the convexity of $t\mapsto \exp(t^{r})$ for $r\geq 1$, we have
\begin{align*}
\E \left[  \exp\left( \Y_*^s \right) \right] & = \E \left[ \exp\left( \E_*' \left[ \Y(\theta'_0(\omega,\omega'))^{d-1-\delta} \right]^{s/(d-1-\delta)} \right) \right]  \\
& \leq \E \left[ \E_*' \left[ \exp\left( \Y(\theta'_0(\omega,\omega'))^{s} \right) \right] \right]  \\
& = \E \left[  \exp\left( \Y^s \right) \right].
\end{align*}
Integrability of lower moments for $s\in (0, d-1-\delta)$ follows by the bound
\begin{equation*}
\E \left[  \exp\left( \Y_*^s \right) \right] =\E \left[  \exp\left( \Y_*^{d-1-\delta} \right) \right] \sup\left|\exp\left( -\Y_{*}^{d-1-\delta}+\Y_{*}^{s}\right)\right|\leq \E \left[  \exp\left( \Y_*^{d-1-\delta} \right) \right]
\end{equation*}
by the monotonicity of the map $p\mapsto x^{p}$ for $p\geq 0$ and $x\geq 1$, which we can take without loss of generality by letting $\Y_{*}=\Y_{*}+1$. We may now redefine $\X$ to be $\X+\Y_*$ to get one side of the desired bound~\eqref{e.Vbound1}. The analogous bound from below is obtained by exchanging $M$ for $-M$ in the equation for $\phi_\ep$, or by repeating the above argument and comparing~$w$ to~$-G_\ep'$.

\smallskip

The main reason that this argument fails to be rigorous is technical: the quantity $\left[ \phi_\ep \right]_{C^{1,1}( B_{\ell/2}(z))}$ is not actually controlled by Theorem~\ref{t.regularity}, rather we have control only over the coarsened quantity $\left[ \phi_\ep \right]_{C^{1,1}_{1}( B_{\ell/2}(z))}$. Most of the work in the proof of Proposition~\ref{p.sensitivity} is therefore to fix this glitch by proving that~\eqref{e.rescomp} still holds if we replace the H\"older seminorm on the right side by the appropriate coarsened seminorm. This is handled by the following lemma, which we write in a rather general form:\begin{lemma}
\label{l.coarseABP}
Let $\ep \in (0,\frac{1}{\ell}]$ and $z\in \Zd$ and suppose $A,A'\in \Omega$ satisfy~\eqref{e.AtotA}. Also fix  $f,f'\in C(\Rd) \cap L^\infty(\Rd)$ and let $u,u'\in C(\Rd) \cap L^{\infty}(\mathbb{R}^{d}) $ be the solutions of
\begin{equation*} \label{}
\left\{ \begin{aligned}
& \ep^2 u - \tr\left(A(x) D^2u\right) = f & \mbox{in} & \ \Rd, \\
& \ep^2 u' - \tr\left(A'(x) D^2u'\right) = f' & \mbox{in} & \ \Rd.
\end{aligned} \right.
\end{equation*}
Then there exists~$C(d,\lambda,\Lambda,\ell)\geq 1$ such that, for every $x\in \Rd$ and $\delta\in (0,1]$,
\begin{multline} \label{e.coarseABP}
u(x) - u'(x) \\
\leq C\left( \delta +  \left[ u \right]_{C^{1,1}_1(z,B_{1/\ep}(z))} + \sup_{y\in B_\ell(z), \, y'\in\Rd} \left( f(y') - f(y) - \delta \ep^{2}|y'-z|^2 \right) \right)G_\ep'(x,z) \\ + \sum_{y\in \Zd} G_\ep(x,y) \sup_{B_{\ell/2}(y)} (f-f').
\end{multline}
Here $G_\ep'$ denotes the modified Green's function for $A'$. 
\end{lemma}

\begin{proof}
We may assume without loss of generality that $z=0$. By replacing $u(x)$ by the function $$u''(x):= u(x)-\sum_{y\in\Zd} G_\ep(x,y) \sup_{B_{\ell/2}(y)} (f-f'),$$ we may assume furthermore that~$f' \geq f$ in~$\Rd$. Fix $\ep\in(0,\tfrac12]$. We will show that
\begin{equation*} \label{}
\sup_{x\in \Rd} \left( u(x) - u'(x) - K G_\ep'(x,0) \right) > 0
\end{equation*}
implies a contradiction for 
\begin{equation*} \label{}
K > 
C\left( \delta +  \left[ u \right]_{C^{1,1}_1(z,B_{1/\ep})} + \sup_{y\in B_\ell, \, y'\in\Rd} \left( f(y') - f(y) - \delta \ep^{2}|y'|^2 \right) \right)G_\ep'(x,0)
\end{equation*}
and $C=C(d,\lambda,\Lambda,\ell)$ chosen sufficiently large.

\smallskip

\emph{Step 1.} We find a touching point $x_0 \in B_{\ell/2}$. Consider the auxiliary function
\begin{equation*} \label{}
\xi(x):= u(x) - u'(x) - K G_\ep'(x,0).
\end{equation*}
By~\eqref{e.AtotA} and using that $f'\geq f$, we see that $\xi$ satisfies
\begin{equation*} \label{}
\ep^2 \xi - \tr\left(A(x) D^2\xi\right)  \leq 0 \quad \mbox{in} \ \Rd \setminus B_{\ell/2}.
\end{equation*}
By the maximum principle and the hypothesis, $\sup_{\Rd} \xi = \sup_{B_{\ell/2}} \xi > 0$. Select $x_0\in B_{\ell/2}$ such that 
\begin{equation} \label{e.touchx0}
\xi(x_0) =  \sup_{\Rd} \xi.
\end{equation}

\emph{Step 2.} We replace $u$ by a quadratic approximation in $B_{\ell}$ and get a new touching point. Select $p\in \Rd$ such that 
\begin{equation} \label{e.innerp1}
\sup_{x\in B_{\ell} } \left| u(x) - u(x_{0}) - p\cdot (x-x_{0})\right| \leq 4\ell^2 \left[ u \right]_{C^{1,1}_1(0,B_{1/\ep})}.
\end{equation}
Fix $\nu\geq1$ to be chosen below and define the function
\begin{equation*} \label{}
\psi(x):= u(x_0) + p\cdot (x-x_0)  - \nu |x-x_0|^2 - u'(x) - KG_\ep'( x,0),
\end{equation*}
The claim is that 
\begin{equation} \label{e.touchclm}
x \mapsto  \psi(x) \quad \mbox{has a local maximum in \ $B_{\ell}$.}
\end{equation}
To verify~\eqref{e.touchclm}, we check that $\psi(x_0) > \sup_{\partial B_{\ell}} \psi$. For~$y\in \partial B_{\ell}$, we compute
\begin{align*}
\psi(x_0) & = u(x_0) - u'(x_0) - KG_\ep'(x_0,0) && \\
& \geq u(y) - u'(y) - KG_\ep'(y,0) && \mbox{(by~\eqref{e.touchx0})} \\
& \geq u(x_0) +p\cdot(y-x_0) - 8\ell^2 \left[ u \right]_{C^{1,1}_1(0,B_{1/\ep})}   - u'(y) - KG_\ep'(y,0)
&& \mbox{(by~\eqref{e.innerp1})} \\
& = \psi(y) + \nu|y-x_0|^2 - 8\ell^2\left[ u \right]_{C^{1,1}_1(0,B_{1/\ep})} \\
& \geq \psi(y) + \ell^2 \nu - 8\ell^2\left[ u \right]_{C^{1,1}_1(0,B_{1/\ep})},
\end{align*}
where in the last line we used $|y-x_0| \geq\frac{\ell}{2}$. Therefore, for every 
\begin{equation*} \label{}
\nu > 8 \left[ u \right]_{C^{1,1}_1(0,B_{1/\ep})},
\end{equation*}
the claim~\eqref{e.touchclm} is satisfied.

\smallskip

\emph{Step 3.} We show that, for every $x\in\Rd$, 
\begin{equation} 
\label{e.ufdirtybound}
u(x) \leq C\delta \ep^{-2} + |x|^2 + \sup_{y\in\Rd} \left( \ep^{-2} f(y) - \delta |y|^2 \right) 
\end{equation}
Define
\begin{equation*} \label{}
w(x):= u(x) - \left( |x|^2 + L \right), \quad \mbox{for} \quad L:= \sup_{x\in\Rd} \left( \ep^{-2} f(x) - \delta |x|^2 \right) + 2d\Lambda \ep^{-2} \delta. 
\end{equation*}
Using the equation for $u$, we find that 
\begin{equation*} \label{}
\ep^2 w - \tr\left(A(x) D^2w \right) 
\leq
f - \ep^2 ( |x|^2 + L) + 2d\Lambda\delta.
\end{equation*}
Using the definition of $L$, we deduce that  
\begin{equation*} \label{}
\ep^2 w - \tr\left(A(x) D^2w \right) \leq 0 \quad \mbox{in} \ \Rd. 
\end{equation*}
Since $w(x) \to -\infty$ as $|x|\to \infty$, we deduce from the maximum principle that $w\leq 0$ in $\Rd$. This yields~\eqref{e.ufdirtybound}.

\smallskip

\emph{Step 4.} We conclude by obtaining a contradiction to~\eqref{e.touchclm} for an appropriate choice of~$K$. Observe that, in $B_\ell$, the function $\psi$ satisfies
\begin{align*}
\ve^{2}\psi-\tr \left( A' (x)D^2 \psi \right)&\leq \ve^{2}(u(x_{0})+p\cdot (x-x_{0})-\nu|x-x_{0}|^{2})+C\nu- f'(x)-K\\
& \leq \ep^2 u(x) + C\nu - f(x) - K \\
& \leq C\nu + \sup_{y\in\Rd} \left( \delta+ f(y) - f(x) - \delta  \ep^{2}|y|^2 \right) - K. 
\end{align*}
Thus~\eqref{e.touchclm} violates the maximum principle provided that
\begin{equation*} \label{}
K > C(\nu+\delta) + \sup_{x\in B_\ell, \, y\in\Rd} \left( f(y) - f(x) -  \delta \ep^{2}|y|^2 \right).
\end{equation*}
This completes the proof. 
\end{proof}

We now use the previous lemma and the estimates in Sections~\ref{s.reg} and~\ref{s.green} to prove the sensitivity estimates~\eqref{e.Vbound1}.

\begin{proof}[Proof of Proposition~\ref{p.sensitivity}]
Fix $s\in (0,d)$. Throughout, $C$ and $c$ will denote positive constants depending only on $(s,d,\lambda,\Lambda,\ell)$ which may vary in each occurrence, and $\X$ denotes an $\F_*$--measurable random variable on $\Omega_*$  satisfying
\begin{equation*} \label{}
\E \left[ \exp\left( \X^p \right) \right] \leq C \quad \mbox{for every} \ p < s,
\end{equation*}
which we also allow to vary in each occurrence.

We fix $z\in \Zd$ and identify the conditional expectation with respect to~$\F_*(\Zd\setminus \{z\})$ via resampling, as in~\eqref{e.resampling}. By the discussion following the statement of the proposition, it suffices to prove the bound~\eqref{e.wtssense} with $\Y_*$ defined by~\eqref{e.Ysrv} for some random variable $\Y \leq \X$. To that end, fix $\omega\in\Omega$, $\omega\in\Omega'$, and denote $\tilde \omega:=\theta_z(\omega,\omega')$ as well as $A=\pi(\omega)$ and $A'=\pi(\tilde \omega)$. Note that~\eqref{e.AtotA} holds. Also let $\phi_\ep$ and $\phi'_\ep$ denote the approximate correctors and $G_\ep$ and $G_\ep'$ the Green's functions. 
\smallskip

\emph{Step 1.} We use Theorem~\ref{t.regularity} to estimate $\left[ \phi_\ep \right]_{C^{1,1}_1(z,B_{1/\ep}(z))}$. In preparation, we rewrite the equation for $\phi_\ep$ in terms of 
\begin{equation*} \label{}
w_\ep(x): = \frac{1}{2}x\cdot Mx  + \phi_\ep(x),
\end{equation*}
which satisfies
\begin{equation*} \label{e.modmodcorr}
-\tr\left( A(x) D^2w_\ep \right) =-\ep^2 \phi_\ep(x). 
\end{equation*}
In view of the fact that the constant functions $\pm\sup_{x\in \mathbb{R}^{d}} \left|\tr(A(x)M)\right|$ are super/subsolutions, we have
\begin{equation} \label{e.detep2bnd}
\sup_{x\in\Rd} \ep^2 \left| \phi_\ep(x)\right| \leq \sup_{x\in\Rd} \left| \tr(A(x)M) \right| \leq d\Lambda \leq C,
\end{equation}
and this yields
\begin{equation} \label{e.woscbndd}
\ep^{2} \osc_{B_{4/\ep}} w_\ep \leq \ep^{2} \osc_{ B_{4/\ep}} \frac{1}{2}x\cdot Mx + \ep^{2} \osc_{B_{4/\ep}} \phi_\ep \leq C.
\end{equation}
By the Krylov-Safonov Holder estimate \eqref{e.classicKS} applied to $w_\ep$ in $B_R$ with $R=4\ep^{-1}$ yields
\begin{equation} \label{e.Holder.scaledapp}
\ep^{2-\beta} \left[ w_\ep \right]_{C^{0,\beta}(B_{2/\ep})} \leq C.
\end{equation}
 Letting $Q(x):= \frac{1}{2}x\cdot Mx$, we have
\begin{equation}
\label{e.phiepcbeta}
\ep^ 2\left[ \phi_\ep  \right]_{C^{0,\beta}(B_{2/\ep})} \leq \ep^2 \left[ w_\ep \right]_{C^{0,\beta}(B_{2/\ep})} +\ep^2 \left[ Q \right]_{C^{0,\beta}(B_{2/\ep})} \leq C\ep^{\beta} + C|M|\ep^{\beta} \leq C\ep^{\beta}.
\end{equation}
We now apply Theorem~\ref{t.regularity} (specifically \eqref{e.pwcC11}) to $w_{\ep}$ with $R=2\ep^{-1}$ to obtain
\begin{align*} \label{}
\left[ w_\ep \right]_{C^{1,1}_{1}(z,B_{1/\ep}(z))} & \leq (T_z\X)^2 \left( \sup_{B_{2/\ep}} \ep^2 \phi+ \ep^{-\beta} \left[ \ep^ 2\phi_\ep  \right]_{C^{0,\beta}(B_{2/\ep})} + \ep^2 \osc_{B_{2/\ep}} w_\ep\right) \\
& \leq C (T_z\X)^2.
\end{align*}
As $w_\ep$ and $\phi_\ep$ differ by the quadratic $Q$, we obtain
\begin{equation} \label{e.C11bndcorr}
\left[ \phi_\ep \right]_{C^{1,1}_{1}(z,B_{1/\ep}(z))} \leq C \left( (T_z\X)^2 + |M| \right) \leq C  (T_z\X)^2. 
\end{equation}
\smallskip

\emph{Step 2.} We estimate~$G_\ep'(\cdot,z)$ and complete the argument for~\eqref{e.wtssense}. By Proposition~\ref{p.Green}, we have
\begin{equation} \label{e.Gwrapp}
G_\ep'(x,z)  \leq (T_z\X)^{d-1-\delta}  \xi_\ep(x-z)\quad
\end{equation}
for $\xi_\ep(x-z)$ defined as in~\eqref{e.dimsep}.
Lemma~\ref{l.coarseABP} yields, for every $x\in\Rd$,
\begin{equation*} \label{}
\big| \phi_\ep(x) - \phi_\ep'(x)\big| \leq C\left( 1 + \left[ \phi_\ep\right]_{C^{1,1}_{1}(z,B_{1/\ep}(z))} \right) G_\ep'(x,z)+2d\Lambda G_{\ep}(x,z).
\end{equation*}
Inserting~\eqref{e.C11bndcorr} and~\eqref{e.Gwrapp} gives
\begin{equation} 
\label{e.phiepprime}
\big| \phi_\ep(x) - \phi_\ep'(x)\big| \leq C (T_z\X)^2 (T_z\X)^{d-1-\delta}  \xi_\ep(x-z).
\end{equation}
This is~\eqref{e.wtssense}.
\end{proof}

\section{Optimal scaling for the approximate correctors}
\label{s.optimals}

We complete the rate computation for the approximate correctors $\phi_{\ep}$. We think of breaking up the decay of $\ep^{2}\phi_{\ep}(0)-\tr(\overline{A}M)$ into two main contributions of error: 
\begin{equation*} \label{}
 \ep^2 \phi_\ep(0) - \tr(\overline{A}M) = \underbrace{\ep^2 \phi_\ep(0) - \E \left[ \ep^2 \phi_\ep(0) \right]}_{\mbox{\small``random error"}} + \underbrace{\E \left[ \ep^2 \phi_\ep(0) \right] - \tr(\overline{A}M)}_{\mbox{\small``deterministic error"}}.
\end{equation*}

 The ``random error" will be controlled by the concentration inequalities established in~Section~\ref{s.sensitivity}. We will show that the ``deterministic error" is controlled by the random error, and this will yield a rate for $\ep^{2}\phi_{\ep}(0)+\tr(\overline{A}M)$. 
 
 \smallskip
 
 First, we control the random error using Proposition~\ref{p.concentration} and the estimates from the previous three sections.
 
 \begin{proposition}
\label{p.randomerror}
There exist $\delta(d,\lambda,\Lambda)>0$ and $C(d,\lambda,\Lambda,\ell) \geq 1$ such that, for every $\ep \in (0,\tfrac12]$,  and $x\in \mathbb{R}^{d}$,
\begin{equation} \label{e.randomerror}
\E \left[  \exp \left(\left( \frac1{\err(\ep)} \left| \ep^2 \phi_\ep(x) - \E \left[\ep^2 \phi_\ep(x) \right] \right| \right)^{\frac{1}{2} + \delta } \right) \right] \leq C.
\end{equation}
\end{proposition}

\begin{proof}
For readability, we prove \eqref{e.randomerror} for~$x=0$. The argument for general $x\in\Rd$ is almost the same. 
Define 
\begin{equation*} \label{}
\xi_\ep(x):= \exp(-a\ep|x|)\cdot \left\{  \begin{aligned} 
 & \log\left( 2 + \frac{1}{\ep(1+|x|)}\right) && \mbox{if} \ d = 2 , \\
 & (1 + |x|)^{2-d}&& \mbox{if} \ d\geq 3.
 \end{aligned} \right.
\end{equation*}
According to Proposition~\ref{p.sensitivity}, for $\beta>0$, 
\begin{align*} \label{}
\lefteqn{ \exp\left(  C \left( \V_*\left[  \frac{\ep^2 \phi_\ep(0)}{\err(\ep)} \right] \right)^\beta  \right)  } \qquad & \\
& = \exp\left(  C \left( \frac{\ep^4}{\err(\ep)^2} \sum_{z\in\Zd} \left( \phi_\ep(0) - \E_*\!\left[ \phi_\ep(0) \,\vert\, \F_*(\Zd\setminus \{ z \}) \right] \right)^2 \right)^\beta  \right) \\
& \leq \exp\left(  C \left( \frac{\ep^4}{\err(\ep)^2} \sum_{z\in\Zd}  (T_z\X)^{2d+2-2\delta} \xi_\ep(z)^2  \right)^\beta  \right).
\end{align*}
We claim (and prove below) that 
\begin{equation}\label{e.randerrwts} 
\sum_{z\in\Z^{d}} \xi_\ep(z)^{2}\leq C\ep^{-4}\err(\ep)^{2}.
\end{equation}
Assuming~\eqref{e.randerrwts} and applying Jensen's inequality for discrete sums, we have
\begin{align*} \label{}
\lefteqn{
\exp\left(  C \left( \frac{\ep^4}{\err(\ep)^2} \sum_{z\in\Zd}  (T_z\X)^{2d+2-2\delta} \xi_\ep(z)^2  \right)^\beta  \right) 
} \qquad & \\
&\leq \exp\left(C\left(\frac{\sum_{z\in\Zd}  (T_z\X)^{2d+2-2\delta} \xi(z)^2  }{\sum_{z\in\Zd} \xi_\ep(z)^{2}}\right)^\beta \right) \\
&\leq \frac{\sum_{z\in\Zd}\xi(z)^2  \exp \left(C\left(T_z\X\right)^{(2d+2-2\delta)\beta}\right) }{\sum_{z\in\Zd} \xi_\ep(z)^{2}}.
\end{align*}
Select $\beta := d/(2d+2- 3\delta)$. Taking expectations, using stationarity, and applying Proposition~\ref{p.sensitivity}, we obtain 
\begin{equation*} \label{}
\E_*\left[ \exp\left(  C \left( \V_*\left[  \frac{\ep^2 \phi_\ep(0)}{\err(\ep)} \right] \right)^\beta  \right) \right] \leq C.
\end{equation*}
Finally, an application of Proposition~\ref{p.concentration} gives, for $\gamma : = 2\beta / (1+\beta) \in (0,2)$,
\begin{equation*} \label{}
\E\left[ \exp\left(  \left| \frac{\ep^2 \phi_\ep(0)}{\err(\ep)} -\E\left[  \frac{\ep^2 \phi_\ep(0)}{\err(\ep)} \right] \right|^\gamma \right) \right] \leq C.
\end{equation*}

This completes the proof of the proposition, subject to the verification of~\eqref{e.randerrwts}, which is a straightforward computation. In dimension $d\geq 3$, we have
\begin{align*}
\sum_{z\in\Zd} \xi_\ep(z)^2  & \leq C \int_{\Rd} \left( 1+ |x| \right)^{4-2d} \exp\left(-2a\ep|x| \right) \, dx \\
& =  C \ep^{d-4} \int_{\Rd} \left( \ep + |y| \right)^{4-2d}\exp\left(-2a|y| \right)\, dy \\
& \leq C\ep^{d-4} \left( \int_{\Rd \setminus B_\ep} |y|^{4-2d} \exp\left(-2a|y| \right) \, dy + \int_{B_\ep} \ep^{4-2d}\, dy \right)  \\
& = C \cdot \left\{ \begin{aligned} & 1 + \ep^{d-4} && \mbox{in} \ d\neq 4,\\
& 1+ \left| \log \ep \right| && \mbox{in}  \  d=4,
\end{aligned} \right. \\
&= C \ep^{-4} \err(\ep)^2.
\end{align*}
In dimension $d=2$, we have
\begin{align*}
\sum_{z\in\Zd} \xi_\ep(z)^2  & \leq C \int_{\Rd} \log^2\left( 2 + \frac{1}{\ep(1+|x|)} \right)  \exp\left(-2a\ep|x| \right) \, dx \\
& \leq C \ep^{-2} \int_{\Rd} \log^2\left(2+\frac{1}{\ep+|y|} \right) \exp\left(-2a|y| \right) \, dy \\
& \leq C \ep^{-2} \bigg(  \int_{B_\ep} \log^2\left( 2+\frac{1}{\ep}\right) \exp\left(-2a|y| \right)\,dy  \\
& \qquad \qquad \qquad + \int_{\Rd \setminus B_{\ep}} \log^2\left( 2+\frac{1}{|y|}\right) \exp\left(-2a|y| \right)  \bigg).
\end{align*}
We estimate the two integrals on the right as follows:
\begin{equation*} \label{}
\int_{B_\ep} \log^2\left( 2+\frac{1}{\ep}\right) \exp\left(-2a|y| \right)\,dy \leq |B_\ep| \log^2\left( 2+\frac{1}{\ep}\right) \leq C \ep^2  \log^2\left( 2+\frac{1}{\ep}\right)
\end{equation*}
and
\begin{multline*}
\int_{\Rd \setminus B_{\ep}} \log^2\left( 2+\frac{1}{|y|}\right) \exp\left(-2a|y| \right)\,dy\\
\leq |B_1| \log^2\left(2+\frac{1}{\ep} \right)  + C\int_{\Rd\setminus B_1} \exp\left( -2a|y| \right)\, dy \leq C\log^2\left(2+\frac1\ep \right).
\end{multline*}
Assembling the last three sets of inequalities yields in $d=2$ that 
\begin{equation*} \label{}
\sum_{z\in\Zd} \xi_\ep(z)^2 \leq C \ep^{-2} \log^2\left(2+\frac1\ep \right) \leq C\ep^{-4} \err(\ep)^2. \qedhere
\end{equation*}
\end{proof}

Next, we show that the deterministic error is controlled from above by the random error. The basic idea is to argue that if the deterministic error is larger than the typical size of the random error, then this is inconsistent with the homogenization. The argument must of course be quantitative, so it is natural that we will apply~Proposition~\ref{p.subopt}. Note that if we possessed the bound $\sup_{x\in\Rd}|\phi_\ep(x) - \E\left[\phi_\ep(0)\right]| \lesssim \ep^{-2} \err(\ep)$, then our proof here would be much simpler. However, this bound is too strong-- we do not have, and of course cannot expect, such a uniform estimate on the fluctuations to hold-- and therefore we need to cut off larger fluctuations and argue by approximation. This is done by using the Alexandrov-Backelman-Pucci estimate and~\eqref{e.randomerror} in a straightforward way.

\begin{proposition}
\label{p.deterministicerror}
There exists $C(d,\lambda,\Lambda,\ell)\geq1$ such that, for every  $\ep \in (0,\tfrac12]$ and $x\in\Rd$,
\begin{equation} \label{e.determinserr}
\left|  \E \left[ \ep^2 \phi_\ep(x) \right] - \tr(\overline{A}M )  \right| \leq C \err(\ep).
\end{equation}
\end{proposition}
\begin{proof}
By symmetry, it suffices to prove the following one-sided bound: for every $\ep\in (0, \frac{1}{2}]$ and $x\in\Rd$,
\begin{equation} \label{e.systwts}
\tr(\overline{A}M)-\E \left[ \ep^2 \phi_\ep(x) \right]  \geq -C \err(\ep).
\end{equation}
The proof of \eqref{e.systwts} will be broken down into several steps. 

\smallskip

\emph{Step 1.} We show that 
\begin{equation} \label{e.offgrid}
\E \left[ \exp\left( \left( \frac{1}{\ep^{-2} \mathcal E(\ep)}\left( \osc_{B_{\sqrt{d}}} \phi_\ep\right)\right)^{\frac12+\delta}\right) \right] \leq C.
\end{equation}
Let $k \in\L$ be the affine function satisfying
\begin{equation*} \label{}
\sup_{x\in B_{\sqrt{d}}} \left| \phi_\ep - k(x) \right| = \inf_{l \in\L} \sup_{x\in B_{\sqrt{d}}} \left| \phi_\ep - l(x) \right|.
\end{equation*}
According to~\eqref{e.C11bndcorr}, 
\begin{equation} 
\label{e.canoteherc11}
\sup_{x\in B_{\sqrt{d}}} \left| \phi_\ep - k(x) \right| \leq C\X^2. 
\end{equation}
Since $k$ is affine, its slope can be estimated by its oscillation on $B_{\sqrt{d}} \cap \Zd$:
\begin{equation*} \label{}
\left| \nabla k \right| \leq C \osc_{B_{\sqrt{d}}\cap \Zd} k.
\end{equation*}
The previous line and~\eqref{e.canoteherc11} yield that  
\begin{equation*} \label{}
\left| \nabla k \right| \leq C \osc_{B_{\sqrt{d}}\cap \Zd} \phi_\ep  + C \X^2. 
\end{equation*}
By stationarity and~\eqref{e.randomerror}, we get 
\begin{equation} 
\label{e.bghlas}
\E \left[ \exp\left( \left( \frac{1}{\ep^{-2} \mathcal E(\ep)}\left( \osc_{B_{\sqrt{d}}\cap \Zd} \phi_\ep\right)\right)^{\frac12+\delta}\right) \right] \leq C
\end{equation}
Therefore, 
\begin{equation*} \label{}
\E \left[ \exp\left( \left( \frac{1}{\ep^{-2} \mathcal E(\ep)}\left| \nabla k \right|\right)^{\frac12+\delta}\right) \right] \leq C
\end{equation*}
The triangle inequality,~\eqref{e.canoteherc11} and~\eqref{e.bghlas} imply~\eqref{e.offgrid}. 

\smallskip

\emph{Step 2.} Consider the function
\begin{equation*} \label{}
f_\ep (x): = \left( -\ep^2 \phi_\ep(x) +  \E\left[\ep^2 \phi_\ep(0) \right] \right)_+.
\end{equation*}
We claim that, for every $R\geq 1$,
\begin{equation} \label{e.ABPsetup}
\E \left[ \left( \fint_{B_R}  \left| f_\ep(x) \right|^d \, dx \right)^{\frac1d} \right] \leq C \err(\ep).
\end{equation}
Indeed, by Jensen's inequality,~\eqref{e.randomerror} and~\eqref{e.offgrid},
\begin{equation*} \label{}
\E \left[ \left( \fint_{B_R}  \left| f_\ep(x) \right|^d \, dx \right)^{\frac1d} \right] \leq \left( \fint_{B_R} \E \left[ \left| f_\ep(x) \right|^d \right]\, dx \right)^{\frac1d} \leq C\err(\ep). 
\end{equation*}

\smallskip

\emph{Step 3.} We prepare the comparison. Define
\begin{equation*} \label{}
\widehat f_\ep(x):= \min\left\{ -\ep^2 \phi_\ep(x), \E \left[ -\ep^2\phi_\ep(0) \right]  \right\} = -\ep^2 \phi_\ep(x) - f_\ep(x),
\end{equation*}
fix $R\geq 1$ (we will send $R\to \infty$ below) and denote by $\widehat \phi_\ep$, the solution of
\begin{equation*}
\left\{ \begin{aligned} 
& -\tr\left( A(x) (M+D^2 \widehat \phi_\ep )\right) = \widehat f_\ep  & \mbox{in} & \ B_R, \\
& \widehat \phi_\ep = d\Lambda \ep^{-2} |M|& \mbox{on} & \ \partial B_R.
\end{aligned} \right.
\end{equation*}
Note that the boundary condition was chosen so that $\phi_\ep \leq \widehat\phi_\ep$ on $\partial B_R$. Thus the Alexandrov-Backelman-Pucci estimate and~\eqref{e.ABPsetup} yield
\begin{equation} \label{e.ABPapp}
\E \left[ R^{-2} \sup_{B_R} \left( \phi_\ep - \widehat\phi_\ep \right)  \right] \leq C \err(\ep).
\end{equation}

\smallskip

\emph{Step 4.} Let $\widehat{\phi}$ denote the solution to
\begin{equation} \label{e.homogdetcmp}
 \left\{ \begin{aligned} 
& -\tr(\overline{A} (M+D^{2}\widehat \phi)) = \E \left[ -\ep^2\phi_\ep(0) \right]  & \mbox{in} & \ B_R, \\
& \widehat \phi = d\Lambda \ep^{-2} |M|& \mbox{on} & \ \partial B_R.
\end{aligned} \right.
\end{equation}
Notice that the right hand side and boundary condition for~\eqref{e.homogdetcmp}~are chosen to be constant. 
Moreover, we can solve for $\widehat \phi$ explicitly: for $x\in B_R$, we have
\begin{equation} \label{e.formhatphi}
\widehat \phi(x) = d\Lambda \ep^{-2} |M| - \frac{|x|^2-R^2}{2\tr \overline{A}}\left(\tr(\overline{A}M) - \E \left[ \ep^2\phi_\ep(0) \right] \right).
\end{equation}
We point out that since $\widehat f_\ep \leq \E \left[ -\ep^2\phi_\ep(0) \right]$, we have that $\widehat\phi_{\ep}$ satisfies 
\begin{equation*} 
 \left\{ \begin{aligned} 
& -\tr(A(x) (M+D^{2}\widehat \phi_{\ep})) \leq \E \left[ -\ep^2\phi_\ep(0) \right]  & \mbox{in} & \ B_R, \\
& \widehat \phi_{\ep} = d\Lambda \ep^{-2}|M| & \mbox{on} & \ \partial B_R.
\end{aligned} \right.
\end{equation*}
It follows then that we may apply Proposition~\ref{p.subopt}~to the pair $\frac{1}{2}x\cdot Mx + \widehat\phi_\ep(x)$ and $\frac{1}{2}x\cdot Mx + \widehat \phi(x)$, which gives
\begin{equation} \label{e.EEapp}
\E \left[ R^{-2} \sup_{x\in B_R} \left( \widehat \phi_\ep - \widehat \phi \right)  \right] \leq CR^{-\alpha}.
\end{equation}

\smallskip

\emph{Step 5.} The conclusion. We have, by~\eqref{e.ABPapp},~\eqref{e.EEapp} and~\eqref{e.formhatphi},
\begin{align*}
-d\Lambda \ep^{-2} |M| &\leq \E \left[ \phi_\ep(0) \right] \\
& \leq \E \left[\widehat \phi_\ep(0) \right] + C \err(\ep) R^2 \\
& \leq \widehat\phi(0) + CR^{2-\alpha} + C \err(\ep) R^2 \\
& = d\Lambda \ep^{-2} |M| + \frac{R^2}{2\tr\overline{A}} \left(\tr(\overline{A}M) - \ep^2\E \left[ \phi_\ep(0) \right] \right) + CR^{2-\alpha} + C \err(\ep) R^2. 
\end{align*} 
Rearranging, we obtain
\begin{equation*} \label{}
\tr(\overline{A}M) - \ep^2\E \left[ \phi_\ep(0) \right] \geq C\left[-2d\Lambda |M| \ep^{-2} R^{-2} - CR^{-\alpha} - C\err(\ep)\right].
\end{equation*}
Sending $R\to \infty$ yields 
\begin{equation*}
\tr(\overline{A}M)-\E \left[ \ep^2 \phi_\ep(0) \right]  \geq -C \err(\ep).
\end{equation*}
Since~\eqref{e.offgrid} and stationarity implies that, for every $x\in\Rd$, 
\begin{equation*} \label{}
\left| \E \left[ \ep^2 \phi_\ep(x) \right] - \E \left[ \ep^2 \phi_\ep(0) \right] \right| \leq C\err(\ep),
\end{equation*}
the proof of~\eqref{e.systwts} is complete.
\end{proof}

The proof of Theorem~\ref{t.correctors} is now complete, as it follows immediately from Propositions~\ref{p.randomerror},~\ref{p.deterministicerror} and~\eqref{e.offgrid}.

\section{Existence of stationary correctors in $d>4$}
\label{s.correctors}

In this section we prove the following result concerning the existence of stationary correctors in dimensions larger than four. 

\begin{theorem}
\label{t.existcorrectors}
Suppose $d>4$ and fix $M\in\Sd$, $|M|=1$. Then there exists a constant $C(d,\lambda,\Lambda,\ell)\geq 1$ and a stationary function $\phi$ belonging $\P$--almost surely to $C(\Rd) \cap L^\infty(\Rd)$, satisfying 
\begin{equation}
\label{e.correctoreq}
-\tr\left(A(x)\left(M+D^2\phi \right) \right) = -\tr\left( \overline AM \right) \quad \mbox{in} \ \Rd
\end{equation}
and, for each $x\in\Rd$ and $t\geq1$,  the estimate
\begin{equation}
\label{e.correctorest}
\P \left[ \left| \phi(x) \right| > t \right]  \leq C \exp\left( -t^{\frac12} \right). 
\end{equation}
\end{theorem}

To prove Theorem~\ref{t.existcorrectors}, we argue that, after subtracting an appropriate constant,~$\phi_\ep$ has an almost sure limit as $\ep \to 0$ to a stationary function~$\phi$. We introduce the functions
\begin{equation*}
\widehat \phi_{\ep}:=\phi_{\ep}-\frac{1}{\ep^{2}}\tr(\overline{A}M)
\end{equation*}
Observe that
\begin{equation}
\label{e.hatphiep}
\ep^{2}\widehat \phi_{\ep}-\tr\left(A(x)D^{2}\widehat \phi_{\ep}\right)=-\tr (\overline{A}M).
\end{equation}
To show that $\hat\phi_\ep$ has an almost sure limit as $\ep\to 0$, we introduce the functions
\begin{equation*} \label{}
\psi_\ep := \phi_\ep - \phi_{2\ep}.
\end{equation*}
Then 
\begin{equation*}
\widehat \phi_{\ep}-\widehat \phi_{2\ep}=\psi_{\ep}-\frac{3}{4 \ep^{2}} \tr (\overline{A}M),
\end{equation*}
and the goal will be to prove bounds on $\widehat \phi_{\ep} - \widehat{\phi}_{\ep}$ which are summable over the sequence $\ep_n:=2^{-n}$. We proceed as in the previous section: we first estimate the fluctuations of~$\psi_\ep$ using a sensitivity estimate and a suitable version of the Efron-Stein inequality. We then use this fluctuation estimate to obtain bounds on its expectation using a variation of the argument in the proof of Proposition~\ref{p.deterministicerror}.

\smallskip

We begin with controlling the fluctuations. 

\begin{lemma}
\label{l.psiepsensitivity}
For every~$p\in [1, \infty)$ and $\gamma>0$, there exists $C(p,\gamma,d,\lambda,\Lambda,\ell)<\infty$ such that, for every $\ep \in\left(0,\frac12\right]$ and $x\in \mathbb{R}^{d}$, 
\begin{equation}\label{e.flucpsi}
\E\left[\left|\psi_{\ep}(x)-\E[\psi_{\ep}(x)]\right|^{p}\right]^{\frac1p}\leq C\ep^{\left(\frac{d-4}2 \wedge 2 \right) -\gamma}.
\end{equation}
\end{lemma}
\begin{proof}
In view of the Efron-Stein inequality for $p$th moments (cf.~\eqref{e.pefronstein}), it suffices to show that 
\begin{equation}\label{e.vertflucpsi}
\E \left[ \V_{*}\left[\psi_{\ep}(x)\right]^{\frac p2} \right]^{\frac1p} \leq C\ep^{\left( \frac{d-4}2 \wedge 2 \right) -\gamma}.
\end{equation}
We start from the observation that $\psi_{\ep}$ satisfies the equation
\begin{equation} \label{e.eq1forpsiep}
\ep^2 \psi_\ep -\tr\left(A(x)D^2\psi_\ep \right) = 3\ep^2\phi_{2\ep}  \quad \mbox{in} \ \Rd.
\end{equation}
Denote the right-hand side by $h_\ep:= 3\ep^2\phi_{2\ep}$. 

\smallskip

\emph{Step 1.} We outline the proof of~\eqref{e.vertflucpsi}. Fix $z\in \mathbb{R}^{d}$. We use the notation from the previous section, letting $A':=\pi(\theta'_z(\omega,\omega'))$ denote a resampling of the coefficients at~$z$. We let $\psi_\ep'$, $\phi_\ep'$, $G_\ep'$, etc, denote the corresponding functions defined with respect to~$A'$. Applying Lemma~\ref{l.coarseABP} with $\delta = \ep^2$, in view of~\eqref{e.eq1forpsiep}, we find that
\begin{align*} 
\lefteqn{
\psi_\ep(x) - \psi'_\ep(x)
} \  & 
\\ & \notag
\leq 
C\left( \ep^2 +  \left[ \psi_\ep \right]_{C^{1,1}_1(z,B_{1/\ep}(z))} + \sup_{y\in B_\ell(z), \, y'\in\Rd} \left( h_\ep(y') - h_\ep(y) - \ep^{4}|y'-z|^2 \right) \right)G_\ep'(x,z) 
\\ & \notag \qquad
+ \sum_{y\in \Zd} G_\ep(x,y) \sup_{B_{\ell/2}(y)} (h_\ep- h_\ep') 
\\ & \notag
 =: CK(z) \xi_\ep(x-z) + C\sum_{y\in\Zd} H(y,z) \xi_\ep( z-y) \xi_\ep(x-y).
\end{align*}
Here we have defined
\begin{multline*}
K(z):= \xi_\ep(x-z)^{-1} \bigg( \ep^2 +  \left[ \psi_\ep \right]_{C^{1,1}_1(z,B_{1/\ep}(z))} \\
+ \sup_{y\in B_\ell(z), \, y'\in\Rd} \left( h_\ep(y') - h_\ep(y) - \ep^{4}|y'-z|^2 \right) \bigg)G_\ep'(x,z)
\end{multline*}
and
\begin{equation*}
H(y,z):=  \left( \xi_\ep( z-y) \xi_\ep(x-y) \right)^{-1} G_\ep(x,y) \sup_{B_{\ell/2}(y)} (h_\ep- h_\ep').
\end{equation*}
These are random variables on the probability space~$\Omega\times\Omega'$ with respect to the probability measure $\tilde{\P} := \P_*\times \P_*'$. Below we will check that, for each $p\in [1, \infty)$ and $\gamma>0$, there exists $C(p,\gamma,d,\lambda,\Lambda,\ell)<\infty$ such that
\begin{equation} \label{e.expKzHyz}
\tilde{\E} \left[ K(z)^p \right]^{\frac1p} 
+
\tilde{\E} \left[ H(y,z)^p \right]^{\frac1p}  \leq C\ep^{2-\gamma}.
\end{equation}
To ease the notation, we will drop the tildes and just write $\E$ instead of $\tilde{\E}$. 
We first complete the proof of~\eqref{e.vertflucpsi} assuming that~\eqref{e.expKzHyz} holds. In view of the discussion in Section~\ref{s.sensitivity}, we compute:
\begin{align*}
\lefteqn{
\E \left[ \V_{*}\left[\psi_{\ep}(x)\right]^{\frac p2} \right]
} \quad & 
\\ &
\leq \E\left[ \left( \sum_{z\in\Zd} \left( CK(z) \xi_\ep(x-z) + C\sum_{y\in\Zd} H(y,z) \xi_\ep( z-y) \xi_\ep(x-y) \right)^2 \right)^{\frac p2} \right] 
\\ &
\leq C \E\left[ \left( \sum_{z\in\Zd} \left[ K(z) \xi_\ep(x-z) \right]^2 \right)^{\frac p2} \right] 
\\ & \qquad
+ C \E\left[ \left( \sum_{z\in\Zd} \left( \sum_{y\in\Zd} H(y,z) \xi_\ep( z-y) \xi_\ep(x-y) \right)^2 \right)^{\frac p2} \right].
\end{align*}
By Jensen's inequality,~\eqref{e.randerrwts} and~\eqref{e.expKzHyz},
\begin{align*}
\E\left[ \left( \sum_{z\in\Zd} \left[ K(z) \xi_\ep(x-z) \right]^2 \right)^{\frac p2} \right] 
& 
\leq \left( \sum_{z\in\Zd} \xi_\ep^2(x-z)\right)^{\frac p2 -1} \E\left[ \left( \sum_{z\in\Zd} K(z)^{p} \xi_\ep^2(x-z)  \right) \right] 
\\ & 
\leq C\ep^{(2-\gamma)p} \left( \sum_{z\in\Zd} \xi_\ep^2(x-z)\right)^{\frac p2} \leq C\ep^{(2-\gamma)p}
\end{align*}
and by Jensen's inequality and~\eqref{e.expKzHyz}, 
\begin{align*}
\lefteqn{
\E\left[ \left( \sum_{z\in\Zd} \left( \sum_{y\in\Zd} H(y,z) \xi_\ep( z-y) \xi_\ep(x-y) \right)^2 \right)^{\frac p2} \right]
} \qquad & 
\\  & 
=  \E\left[ \left( \sum_{z\in\Zd} \sum_{y,y'\in\Zd} H(y,z)H(y',z) \xi_\ep( z-y) \xi_\ep(x-y)\xi_\ep( z-y') \xi_\ep(x-y') \right)^{\frac p2} \right]
\\ & 
\leq \left(  \sum_{z,y,y'\in\Zd} \xi_\ep( z-y) \xi_\ep(x-y)\xi_\ep( z-y') \xi_\ep(x-y') \right)^{\frac p2-1}
 \\ & \ \ \times
 \E\left[\sum_{z,y,y'\in\Zd} H(y,z)^{\frac p2}H(y',z)^{\frac p2} \xi_\ep( z-y) \xi_\ep(x-y)\xi_\ep( z-y') \xi_\ep(x-y') \right]
\\ & 
= \left( \sum_{z\in\Zd} \left( \sum_{y\in\Zd} \xi_\ep( z-y) \xi_\ep(x-y)  \right)^2 \right)^{\frac p2-1}
\\ & \quad 
\times
\sum_{z,y,y'\in\Zd}  \E\left[H(y,z)^{\frac p2}H(y',z)^{\frac p2} \right]\xi_\ep( z-y) \xi_\ep(x-y)\xi_\ep( z-y') \xi_\ep(x-y') 
\\ & 
\leq \left( \sum_{z\in\Zd} \left( \sum_{y\in\Zd} \xi_\ep( z-y) \xi_\ep(x-y)  \right)^2 \right)^{\frac p2} C\ep^{(2-\gamma)p}.
\end{align*}
In view of the inequality 
\begin{equation} \label{e.Greenie1}
\left( \sum_{z\in\Zd} \left( \sum_{y\in\Zd} \xi_\ep( z-y) \xi_\ep(x-y)  \right)^2 \right)^{\frac 12} 
\leq C\left(1+\ep^{-1} \right)^{4-\frac d2},
\end{equation}
which we also will prove below, the demonstration of~\eqref{e.vertflucpsi} is complete. 

\smallskip

To complete the proof, it remains to prove~\eqref{e.expKzHyz} and~\eqref{e.Greenie1}.

\smallskip

\emph{Step 2.} Proof of~\eqref{e.Greenie1}. We first show that, for every $x,z\in \mathbb{R}^{d}$, 
\begin{equation}\label{e.estiny}
\sum_{y\in \mathbb{Z}^{d}} \xi_{\ep}(z-y)\xi_{\ep}(x-y)\leq C\exp(-c\ep|x-z|) \left((1+|x-z|)^{4-d}+\ep^{d-4}\right).
\end{equation}
Recall that in dimensions $d>4$, $\xi_{\ep}(x)=\exp(-a\ep|x|)(1+|x|)^{2-d}$. Denote $r:=|x-z|$. We estimate the sum by an integral and then split the sum into five pieces:
\begin{align*}
\lefteqn{
\sum_{y\in \mathbb{Z}^{d}} \xi_{\ep}(z-y)\xi_{\ep}(x-y)
} \qquad & \\
& \leq C\int_{\mathbb{R}^{d}}\xi_{\ep}(z-y)\xi_{\ep}(x-y)\, dy
\\ & 
\leq  C \int_{B_{r/4}(x)}\xi_{\ep}(z-y)\xi_{\ep}(x-y)\, dy
  +C\int_{B_{r/4}(z)}\xi_{\ep}(z-y)\xi_{\ep}(x-y)\, dy\\
& \quad+C\int_{B_{2r}(z)\setminus \left(B_{r/4}(z)\cup B_{r/4}(x)\right)}\xi_{\ep}(z-y)\xi_{\ep}(x-y)\, dy\\
& \quad+C\int_{B_{1/\ep}(z)\setminus B_{2r}(z)}\xi_{\ep}(z-y)\xi_{\ep}(x-y)\, dy\\
& \quad+C\int_{\mathbb{R}^{d}\setminus B_{1/\ep}(z)}\xi_{\ep}(z-y)\xi_{\ep}(x-y)\, dy.
\end{align*}
We now estimate each of the above terms. Observe first that
\begin{multline*}
\int_{B_{r/4}(x)} \xi_{\ep}(z-y)\xi_{\ep}(x-y)\, dy+\int_{B_{r/4}(z)} \xi_{\ep}(z-y)\xi_{\ep}(x-y)\, dy
\\
\leq C \exp(-c\ep r) \int_{B_{r}(0)} (1+r)^{2-d}(1+|y|)^{2-d}\, dy = C\exp(-c\ep r) (1+r)^{4-d}. 
\end{multline*}
Next, we estimate 
\begin{multline*}
\int_{B_{2r}(z)\setminus \left(B_{r/4}(z)\cup B_{r/4}(x)\right)} \xi_{\ep}(z-y)\xi_{\ep}(x-y)\, dy
\\
\leq C\exp (-c\ep r)\int_{B_{2r}(z)} (1+|y|)^{2(2-d)}\, dy C\exp(-c\ep r) (1+r)^{4-d} 
\end{multline*}
and
\begin{multline*}
\int_{B_{1/\ep}(z)\setminus B_{2r}(z)}\xi_{\ep}(z-y)\xi_{\ep}(x-y)\, dy
\\
\leq C\int_{B_{1/\ep}\setminus B_{2r}} \exp(-c\ep r) (1+|y|)^{4-2d}\, dy 
= C\exp(-c\ep r) (1+r)^{4-d}. 
\end{multline*}
Finally, since $d>4$, 
 \begin{align*}
 \lefteqn{
\int_{\mathbb{R}^{d}\setminus B_{1/\ep}(z)}\xi_{\ep}(z-y)\xi_{\ep}(x-y)\, dy
} \qquad & \\
&\leq C\exp(-c\ep r)\int_{\mathbb{R}^{d}\setminus B_{1/\ep}} \exp(-c\ep |y|)(1+|y|)^{4-2d}\, dy \\
&=C\exp(-c\ep r)\ep^{d-4}\int_{\mathbb{R}^{d}\setminus B_{1}} \exp(-2a|y|)(\ep+|y|)^{4-2d}\, dy \notag\\
&= C\exp(-c\ep r)\ep^{d-4}. \notag
\end{align*}
Combining the above inequalities yields~\eqref{e.estiny}.

\smallskip

To obtain \eqref{e.Greenie1}, we square~\eqref{e.estiny} and sum it over $z\in\Zd$ to find that 
\begin{multline*}
 \sum_{z\in\Zd} \left( \sum_{y\in\Zd} \xi_\ep( z-y) \xi_\ep(x-y)  \right)^2 
 \\
  \leq\int_{\mathbb{R}^{d}}C \exp\left( -c\ep|x-z| \right) \left((1+|x-z|)^{8-2d}+\ep^{2d-8}\right)\, dz =C + C\ep^{d-8}.
\end{multline*}

\smallskip

\emph{Step 3.} The estimate of the first term on the left side of~\eqref{e.expKzHyz}. 
Notice that according to Proposition~\ref{p.Green}, we have that 
\begin{multline}\label{e.betterKest}
|K(z)|\leq (T_{z}\X(\theta'_{z}(\omega, \omega')))^{d-1-\delta}\left( \ep^2 +  \left[ \psi_\ep \right]_{C^{1,1}_1(z,B_{1/\ep}(z))}\right. \\
+\left. \sup_{y\in B_\ell(z), \, y'\in\Rd} \left( h_\ep(y') - h_\ep(y) - \ep^{4}|y'-z|^2 \right)\right).
\end{multline}
We control each part individually. First, we claim that for every $\gamma\in (0,1)$, for every $p\in (1, \infty)$, there exists $C(\gamma, \la, \La, d, \ell, p)$ such that 
\begin{equation}\label{e.c11psi}
\E\left[\left(\left[ \psi_\ep \right]_{C^{1,1}_1(z,B_{1/\ep}(z))}\right)^{p}\right]^{\frac 1p} \leq C\ep^{2-\gamma}.
\end{equation}
Observe that $\psi_{\ep}$ is a solution of
\begin{equation} 
\label{e.psiep2}
 -\tr\left(A(x)D^2\psi_\ep \right) =  -\ep^2\phi_\ep + 4\ep^2\phi_{2\ep} \quad \mbox{in} \ \Rd. 
\end{equation}
Denote the right side by $f_\ep:= -\ep^2\phi_\ep + 4\ep^2\phi_{2\ep}=-\ep^{2}\psi_{\ep}+3\ep^{2}\phi_{2\ep}$.

\smallskip

We show that, for every $\gamma >0$ and $p\in [1,\infty)$, there exists $C(\gamma,p,d,\lambda,\Lambda,\ell)<\infty$ such that, for every $\ep\in (0,\frac12]$,
\begin{equation}
\label{e.fepbounds}
\E \left[ \left( \left\| f_\ep \right\|_{L^\infty(B_{1/\ep})} + \ep^{-\beta} \left[ f_\ep  \right]_{C^{0,\beta}(B_{1/\ep})}  \right)^p\, \right]^{\frac1p} \leq C\ep^{2-\gamma}.
\end{equation}
We first observe that~\eqref{e.randomerror},~\eqref{e.determinserr}, and~\eqref{e.offgrid} imply that 
\begin{equation*} \label{}
\E \left[ \exp\left( \left( \frac{1}{ \mathcal E(\ep)} \sup_{B_{\sqrt{d}} } \left| \ep^2\phi_\ep - \tr\left(\overline{A}M \right) \right| \right)^{\frac12+\delta}\right) \right] \leq C.
\end{equation*}
A union bound and stationarity then give, for every $\gamma>0$ and $p\in [1,\infty)$,
\begin{equation} 
\label{e.newoscboundphiep}
\E \left[  \left\|\ep^2\phi_\ep - \tr\left(\overline{A}M \right) \right\|_{L^\infty(B_{4/\ep})}^p \right]^{\frac1p} \leq C\ep^{-\gamma} \mathcal{E}(\ep).
\end{equation}
where~$C=C(\gamma,p,d,\lambda,\Lambda,\ell)<\infty$. The Krylov-Safonov estimate yields 
\begin{equation*} \label{}
\E \left[  \left( \ep^{-\beta} \left[ \phi_\ep \right]_{C^{0,\beta}(B_{2/\ep})} \right)^p\, \right]^{\frac1p}  \leq \ep^{-2} \mathcal{E}(\ep). 
\end{equation*}
The previous two displays and the triangle inequality yield the claim~\eqref{e.fepbounds}. Now~\eqref{e.c11psi} follows from Theorem~\ref{t.regularity},~\eqref{e.fepbounds},~\eqref{e.newoscboundphiep} and the H\"older inequality. 

\smallskip

We next show that for every $\gamma\in (0,1)$ and for every $p\in (1, \infty)$, 
\begin{equation} \label{e.grosstermpsi}
\E \left[\left(  \sup_{y\in B_\ell(z), \, y'\in\Rd} \left( h_\ep(y') - h_\ep(y) - \ep^{4}|y'-z|^2 \right)\right)^p \right]^{\frac 1p}
\leq C\ep^{2-\gamma}.
\end{equation}
By a union bound, we find that, for every $t>0$,
\begin{align*} \label{}
\lefteqn{
\P \left[  \sup_{y'\in\Rd} \left( h_\ep(y') -\ep^{4}|y'-z|^2 \right) > t   \right] 
} \qquad & \\
& \leq \sum_{n=0}^\infty \P \left[  \sup_{y' \in B_{2^n/\ep}(z)} h_\ep(y') >  c2^{2n}\ep^{2} + t \right] 
\\ &
\leq
\sum_{n=0}^{n(t)} \P \left[  \sup_{y' \in B_{2^n/\ep}(z)} h_\ep(y') >   t \right]
+ \sum_{n=n(t)+1}^{\infty} \P \left[  \sup_{y' \in B_{2^n/\ep}(z)} h_\ep(y') >  c2^{2n}\ep^{2} \right]
\end{align*}
where $n(t)$ is the largest positive integer satisfying $2^{2n(t)} \leq t\ep^{-2}$. By~\eqref{e.newoscboundphiep},
\begin{align*} \label{}
\sum_{n=0}^{n(t)} \P \left[  \sup_{y' \in B_{2^n/\ep}(z)} h_\ep(y') >   t \right] 
&
\leq (n(t)+1)  \P \left[  \sup_{y' \in B_{2^{n(t)}/\ep}(z)} h_\ep(y') >   t\right] 
\\ &
\leq C(n(t)+1) 2^{dn(t)}  \P \left[  \sup_{y' \in B_{1/\ep}} h_\ep(y') >   t\right] 
\\ & 
\leq C(n(t)+1) 2^{dn(t)} \left( \ep^{\gamma-2} t \right)^{-p}
\\ &
\leq C(\log t \ep^{-2})\left(t^{-1}\ep^{(2-\gamma)}\right)^{p-\frac d2}\ep^{\frac{\gamma d}{2}}
\end{align*}
and 
\begin{align*}
 \sum_{n=n(t)+1}^{\infty}\P \left[  \sup_{y' \in B_{2^n/\ep}(z)} h_\ep(y') >  c2^{2n}\ep^{2} \right]
&
\leq
 \sum_{n=n(t)+1}^{\infty} C2^{dn} \P \left[  \sup_{y' \in B_{1/\ep}} h_\ep(y') >  c2^{2n}\ep^{2} \right] 
 \\ &
\leq
C\ep^{(2-\gamma)p}  \sum_{n=n(t)+1}^{\infty} 2^{dn} 2^{-2np}\ep^{-2p}  
\\ &
\leq C\left( \ep^{(2-\gamma)} t^{-1} \right)^{p-\frac d2}.
\end{align*}
Combining the above, taking~$p$ sufficiently large, integrating over~$t$, and shrinking~$\gamma$ and redefining~$p$ yields \eqref{e.grosstermpsi}.

\smallskip

A combination of \eqref{e.bghlas}, \eqref{e.betterKest}, \eqref{e.c11psi},~ \eqref{e.grosstermpsi} and the H\"older inequality yields the desired bound for the first term on the left side of~\eqref{e.expKzHyz}.

\smallskip

\emph{Step 4.} The estimate of the second term on the left side of~\eqref{e.expKzHyz}. 
According to~\eqref{e.phiepprime} and Proposition~\ref{p.Green}, for $\X$ and $\delta(d, \la, \La)>0$ as in Proposition~\ref{p.Green},
\begin{align*} \label{}
G_\ep(x,y) \sup_{B_{\ell/2}(y)} (h_\ep- h_\ep') &\leq C\ep^2(T_{y}\X)^{d-1-\delta}\xi_{\ep}(y-x) (T_{z}\X)^{d+1-\delta}\xi_{2\ep}(y-z)\\
&\leq C\ep^{2} (T_{y}\X)^{d-1-\delta}(T_{z}\X)^{d+1-\delta}\xi_{\ep}(y-x) \xi_{\ep}(y-z).
\end{align*}
Therefore, 
\begin{equation*}
H(y,z)\leq C\ep^{2} (T_{y}\X)^{d-1-\delta}(T_{z}\X)^{d+1-\delta}.
\end{equation*}
Thus H\"older's inequality yields that, for every $p\in(1,\infty)$,
\begin{equation*}
\tilde{\E}\left[H(y,z)^{p}\right]
\leq C\ep^{2p}.
\end{equation*}
This completes the proof of~\eqref{e.expKzHyz}.
\end{proof}

We next control the expectation of $\hat{\phi}_{\ep}-\hat{\phi}_{2\ep} =  \psi_\ep(x)- \frac3{4\ep^2} \tr\left( \overline{A} M \right)$. 

\begin{lemma}
\label{l.expectationpsiep}
For every $p\in [1, \infty)$ and $\gamma>0$, there exists ~$C(p,\gamma,d,\lambda,\Lambda,\ell)<\infty$ such that, for every $\ep\in \left(0,\frac12\right]$ and $x\in\Rd$, 
\begin{equation}
\label{e.Epsiep}
\left| \E[\psi_{\ep}(x)] - \frac3{4\ep^2} \tr\left( \overline{A} M \right) \right|
\leq 
C\ep^{\left(\frac{d-4}2 \wedge 2 \right) -\gamma}.
\end{equation}
\end{lemma}
\begin{proof}
The main step in the argument is to show that 
\begin{equation} 
\label{e.psieptelescoping}
\left| \E\left[ \psi_\ep(0)  \right] - 4\E\left[ \psi_{2\ep}(0)  \right]\right| \leq C\ep^{\frac{d-4}{2}\wedge 2}.
\end{equation}
Let us assume~\eqref{e.psieptelescoping} for the moment and see how to obtain~\eqref{e.Epsiep} from it. First, it follows from~\eqref{e.psieptelescoping} that, for every $\ep$ and $m,n\in\N$ with $m\leq n$, 
\begin{equation*} \label{}
\left| (2^{-m}\ep)^2 \E\left[ \psi_{2^{-m}\ep}(0)  \right] - (2^{-n}\ep)^2 \E \left[ \psi_{2^{-n}\ep}(0) \right] \right| \leq C(2^{-m}\ep)^{\frac d2\wedge 4}.
\end{equation*}
Thus, the sequence $\left\{ (2^{-n}\ep)^2 \E \left[ \psi_{2^{-n}\ep} (0)\right] \right\}_{n\in\N}$ is Cauchy and there exists $L \in\R$ with 
\begin{equation*} \label{}
\left| (2^{-m}\ep)^2 \E\left[ \psi_{2^{-m}\ep}(0)  \right] -L \right| \leq C(2^{-m}\ep)^{\frac d2\wedge 4}.
\end{equation*}
Taking $m=0$ and dividing by $\ep^2$, this yields 
\begin{equation*} \label{}
\left|  \E\left[\psi_{\ep}(0)  \right] - \frac{L}{\ep^2} \right| \leq C\ep^{\frac {d-4}2\wedge 2}.
\end{equation*}
But in view of~\eqref{e.determinserr}, we have that $L=\frac34 \tr\left(\overline{A}M \right)$. This completes the proof of the lemma, subject to the verification of~\eqref{e.psieptelescoping}.

\smallskip

We denote $h_\ep(x) := \psi_\ep (x) - 4\psi_{2\ep}(x)$ so that we may rewrite~\eqref{e.psieptelescoping} as 
\begin{equation} \label{e.psieptelescoping2}
\E\left[ h_\ep(0) \right] \leq C\ep^{\frac{d-4}{2}\wedge 2}.
\end{equation}
We next introduce the function
\begin{equation} \label{e.etaep}
\eta_\ep(x):= \psi_\ep - \frac14 \psi_{2\ep}
\end{equation}
and observe that $\eta_\ep$ is a solution of
\begin{equation} \label{e.eqetaep}
-\tr\left( A(x)D^2\eta_\ep \right) = -\ep^2 h_\ep \quad \mbox{in} \ \Rd. 
\end{equation}
In the first step, we show that $\psi_\ep$ has small oscillations in balls of radius $\ep^{-1}$, and therefore so do $h_\ep$ and $\eta_\ep$. This will allow us to show in the second step that~\eqref{e.eqetaep} is in violation of the maximum principle unless the mean of $h_\ep$ is close to zero.

\smallskip

   \emph{Step 1.} The oscillation bound for~$\psi_\ep$. The claim is that, for every~$\gamma\in (0,1)$ and~$p\in [1, \infty)$, there exists $C(p, \gamma, \la, \La, d, \ell)$ such that 
\begin{equation}\label{e.psiosc}
\E\left[\sup_{x\in B_{1/\ep}} \left| \psi_{\ep}(x) - \E\left[ \psi_\ep(0) \right] \right|^{p}\right]^{\frac1p} \leq C \ep^{\frac{d-4}{2}\wedge 2}.
\end{equation}
By the equation~\eqref{e.psiep2} for~$\psi_\ep$, the Krylov-Safonov estimate~\eqref{e.classicKS} and the bounds~\eqref{e.fepbounds}, we have that (taking $\gamma$ sufficiently small), 
\begin{equation} 
\label{e.snaptogridpsi}
\E \left[ \left( \sup_{x\in B_{1/2\ep}} \left| \psi_\ep(x) - \psi_\ep([x]) \right| \right)^p \right]^{\frac1p} 
\leq C \ep^{\sigma} \ep^{2-\gamma}  \leq C\ep^2. 
\end{equation}
Here $[x]$ denotes the nearest point of $\Zd$ to $x\in\Rd$. 
By the fluctuation estimate~\eqref{e.flucpsi}, stationarity and a union bound, we have, for every $\gamma>0$ and $p\in(1,\infty)$,
\begin{equation*} \label{}
\E\left[\sup_{z\in B_{1/\ep} \cap \Zd} \left| \psi_{\ep}(z) - \E\left[ \psi_\ep(0) \right] \right|^{p}\right]^{\frac1p}
\leq C \ep^{\frac{d-4}{2}\wedge 2-\gamma}.
\end{equation*}
The previous two lines complete the proof of~\eqref{e.psiosc}. 

\smallskip

\emph{Step 2.} We prove something stronger than~\eqref{e.psieptelescoping2} by showing that
\begin{equation*} \label{}
\E \left[ \sup_{x\in B_{1/\ep} } \left| h_\ep (x) \right| \right] \leq C\ep^{\frac{d-4}{2}\wedge 2}.
\end{equation*}
By~\eqref{e.psiosc}, it suffices to show that 
\begin{equation*} \label{}
\E  \left[ \sup_{x\in B_{1/\ep} } h_\ep (x) \right] \geq - C\ep^{\frac{d-4}{2}\wedge 2}
\quad \mbox{and} \quad 
\E  \left[ \inf_{x\in B_{1/\ep} } h_\ep (x) \right] \leq C\ep^{\frac{d-4}{2}\wedge 2}.
\end{equation*}
We will give only the argument for the second inequality in the display above since the proof of the first one is similar. Define the random variable 
\begin{equation*} \label{}
\kappa: = \ep^2 \sup_{x\in B_{1/\ep}} \left| \psi_{\ep}(x) - \E\left[ \psi_\ep(0) \right] \right|.
\end{equation*}
Observe that the function
\begin{equation*} \label{}
x \mapsto \psi_\ep(x) - 8 \kappa |x|^2 \quad \mbox{has a local maximum at some point} \ x_0 \in B_{1/2\ep}. 
\end{equation*}
The equation~\eqref{e.psiep2} for $\psi_\ep$ implies that 
\begin{equation*} \label{}
 -\ep^2 h_\ep(x_0) \geq -16 \Lambda d \kappa \geq -C\kappa.
\end{equation*}
Thus
\begin{equation*} \label{}
\inf_{x\in B_{1/2\ep}} h_\ep(x) \leq h_\ep (x_0) \leq C\kappa = C \sup_{x\in B_{1/\ep}} \left| \psi_{\ep}(x) - \E\left[ \psi_\ep(0) \right] \right|. 
\end{equation*}
Taking expectations and applying~\eqref{e.psiosc} yields the claim. 
\end{proof}

We now complete the proof of Theorem~\ref{t.existcorrectors}. 

\begin{proof}{Proof of~Theorem~\ref{t.existcorrectors}}
According to~\eqref{e.flucpsi},~\eqref{e.psiosc},~\eqref{e.Epsiep} and a union bound, we have
\begin{equation} \label{}
\E \left[ \sup_{x\in B_{1/\ep}} \left| \psi_{\ep}(x) - \frac3{4\ep^2} \tr\left( \overline{A} M \right) \right| \right] 
\leq C \ep^{\left(\frac{d-4}2 \wedge 2 \right) -\gamma}.
\end{equation}
From this we deduce that 
\begin{equation} \label{}
\E \left[ \sup_{x\in B_{1/\ep}} \left| \hat{\phi}_{\ep}(x) - \hat{\phi}_{2\ep} (x) \right| \right] 
\leq C \ep^{\left(\frac{d-4}2 \wedge 2 \right) -\gamma}.
\end{equation}
We deduce the existence of a stationary function $\phi$ satisfying
\begin{equation} \label{}
\E \left[ \sup_{x\in B_{1/\ep}} \left| \hat{\phi}_{\ep}(x) - \phi (x) \right| \right] 
\leq C \ep^{\left(\frac{d-4}2 \wedge 2 \right) -\gamma}.
\end{equation}
Passing to the limit $\ep\to 0$ in~\eqref{e.hatphiep} and using the stability of solutions under uniform convergence, we obtain that $\phi$ is a solution of~\eqref{e.correctoreq}. The estimates~\eqref{e.correctorest} are immediate from~\eqref{e.correctorerror}. This completes the proof of the theorem. 
\end{proof}

\appendix

\section{Proof of the stretched exponential spectral gap inequality}
\label{s.spectralgap}

We give the proof of Proposition~\ref{p.concentration}, the \emph{Efron-Stein-type} inequality for stretched exponential moments. We first recall the classical Efron-Stein (often called the ``spectral gap") inequality. Given a probability space $(\Omega, \mathcal{F}, \mathbb{P})$ and a sequence $\mathcal{F}_{k}\subseteq \mathcal{F}$ of independent $\sigma$-algebras, let $X$ denote a random variable which is measurable with respect to $\mathcal{F}:=\sigma(\mathcal{F}_{1}, \ldots, \mathcal{F}_{n})$. The classical Efron-Stein inequality states that 
\begin{equation}\label{e.efronstein}
\var [X]=\E[\left|X-\E[X]\right|^{2}]\leq \E\left[\sum_{i=1}^{n}(X-X_{i}')^{2}\right]
\end{equation}
where 
\begin{align*}
\mathcal{F}_{i}'&:=\sigma(\mathcal{F}_{1}, \ldots, \mathcal{F}_{i-1}, \mathcal{F}_{i+1}, \ldots, \mathcal{F}_{n})\\
X'_{i}&:=\E\left[X \mid \mathcal{F}_{i}'\right].
\end{align*}
Therefore, we see that the variance is controlled by the $\ell^{2}$-norm of the vertical derivative 
\begin{equation*}
\V[X]:=\sum_{i=1}^{n}(X-X_{i}')^{2}.
\end{equation*}
If we have control of higher moments of $\V[X]$, then we can obtain estimates on the moments of $|X-\E[X]|$. 
Indeed, a result  of Boucheron, Lugosi, and Massart which can be found in~\cite{BLM,BLMbook} states that for every $p\geq 2$, 
\begin{equation}\label{e.pefronstein}
\E\left[\left|X-\E[X]\right|^{p}\right]\leq Cp^{p/2}\E\left[\V[X]^{p/2}\right],
\end{equation}
where we may take $C=1.271$. The same authors were also able to obtain similar estimates on exponential moments of $X-\E[X]$. Their result~\cite{BLM,BLMbook} states that 
\begin{equation}
\E[\exp(|X-\E[X]|)]\leq C\E[\exp(C\V[X])].
\end{equation}

\smallskip

We now give the proof of Proposition~\ref{p.concentration}, which is obtained by writing a power series formula for the stretched exponential and then using~\eqref{e.pefronstein} to estimate each term. We thank J.~C. Mourrat for pointing out this simple argument and allowing us to include it here.
 
\begin{proof}[{Proof of Proposition~\ref{p.concentration}}]
We show that for every $\beta\in (0,2)$, 
\begin{equation*}
\E\left[\exp\left(|X-\E[X]|^{\beta}\right)\right]\leq C\E\left[\exp\left((C\V[X])^{\frac{\beta}{2-\beta}}\right)\right]^{\frac{2-\beta}{\beta}}.
\end{equation*}
We may assume without loss of generality that $\E[X]=0$. Fix $\beta\in(0,2)$. We estimate the power series 
\begin{equation*}
\E\left[\exp\left(|X|^{\beta}\right)\right]=\sum_{n=0}^{\infty}\frac{1}{n!}\E\left[|X|^{\beta n}\right]
\end{equation*}
by splitting up the sum into two pieces and estimating each of them separately as follows. First, we consider the terms in which the power of $|X|$ is less than 2, and apply \eqref{e.efronstein} to get
\begin{multline*}
\sum_{n=0}^{\lfloor 2/\beta\rfloor}\frac{1}{n!}\E\left[|X|^{\beta n}\right]
\leq \sum_{n=0}^{\lfloor 2/\beta\rfloor}\frac{1}{n!}\E\left[|X|^{2}\right]^{n\beta/2}
\leq  \sum_{n=0}^{\lfloor 2/\beta\rfloor}\frac{1}{n!}\E\left[\V[X]\right]^{n\beta/2} \\
\leq \exp(\E\left[\V[X]\right]^{\beta/2}). 
\end{multline*}
For the other terms, we apply \eqref{e.pefronstein} (with $\kappa$ in place of $C$) and the discrete Holder inequality to obtain, for $\alpha>0$ to be selected below, 
\begin{align*}
\sum_{n=\lfloor 2/\beta\rfloor+1}^{\infty}\frac{1}{n!}\E\left[|X|^{\beta n}\right]&\leq \sum_{n=1}^{\infty}\frac{1}{n!}(\kappa n\beta)^{n\beta/2}\E\left[\left(\V[X]\right)^{n\beta/2}\right]\\
&\leq \left(\sum_{n=1}^{\infty}\frac{1}{n!}\left(\frac{\kappa n \beta}{\alpha}\right)^{n}\right)^{\beta/2}\left(\sum_{n=1}^{\infty}\frac{1}{n!}\E\left[(\alpha \V[X])^{n\beta/2}\right]^{\frac{2}{2-\beta}}\right)^{\frac{2-\beta}{2}}.
\end{align*}
We estimate the first factor on the right hand side by using the classical inequality (related to Stirling's approximation) which states that, for every~$n\in \mathbb{N}$, 
\begin{equation*}
n!\geq (2\pi)^{1/2}n^{n+1/2}\exp(-n).
\end{equation*}
This yields that for every $\alpha>e\kappa \beta$, 
\begin{equation*}
\sum_{n=1}^{\infty}\frac{1}{n!}\left(\frac{\kappa n \beta}{\alpha}\right)^{n}\leq \frac{1}{\sqrt{2\pi}}\sum_{n=1}^{\infty}n^{-1/2}\left(\frac{e\kappa\beta}{\alpha}\right)^{n}\leq \frac{\alpha}{\alpha-e\kappa\beta}. 
\end{equation*}
Combining this with our previous estimate, we obtain 
\begin{equation*}
\sum_{n=\lfloor 2/\beta\rfloor+1}^{\infty}\frac{1}{n!}\E\left[|X|^{\beta}n\right]\leq \left(\frac{\alpha}{\alpha-e\kappa\beta}\right)^{\beta/2}\left(\sum_{n=1}^{\infty}\frac{1}{n!}\E\left[(\alpha \V[X])^{n\beta/2}\right]^{\frac{2}{2-\beta}}\right)^{\frac{2-\beta}{2}}.
\end{equation*}
Observe that 
\begin{equation*}
\sum_{n=1}^{\infty}\frac{1}{n!}\E\left[(\alpha \V[X])^{n\beta/2}\right]^{\frac{2}{2-\beta}}\leq \sum_{n=1}^{\infty}\frac{1}{n!} \E\left[(\alpha \V[X])^{\frac{n\beta}{2-\beta}}\right]=\E\left[\exp((\alpha\V[X])^{\frac{\beta}{2-\beta}}\right],
\end{equation*}
and this implies that 
\begin{equation*}
\sum_{n=\lfloor 2/\beta\rfloor+1}^{\infty}\frac{1}{n!}\E\left[|X|^{\beta}n\right]\leq \left(\frac{\alpha}{\alpha-e\kappa\beta}\right)^{\beta/2}\left(\E\left[\exp((\alpha\V[X])^{\frac{\beta}{2-\beta}}\right]\right)^{\frac{2-\beta}{2}}.
\end{equation*}
Combining all of the previous estimates yields that 
\begin{equation*}
\E\left[\exp\left(|X|^{\beta}\right)\right]\leq \exp\left(\E[\V[X]]^{\beta/2}\right)+C\E\left[\exp\left((C\V[X])^{\frac{\beta}{2-\beta}}\right)\right]^{\frac{2-\beta}{2}}.
\end{equation*}
Since $\beta\leq 2$ and we can take $\alpha=20$, the constant $C$ is universal. This completes the proof. 
\end{proof}

\subsection*{Acknowledgements}
The second author was partially supported by NSF Grant DMS-1147523.

\bibliographystyle{plain}
\bibliography{optimalrates}
\end{document}